\documentclass[10pt]{amsart}
\usepackage[latin1]{inputenc}
\usepackage{amsmath}
\usepackage{amssymb}
\usepackage{latexsym}
\usepackage{verbatim}
\usepackage{graphics}
\usepackage{multirow}
\usepackage{amsfonts}
\usepackage{graphicx}
\usepackage{epstopdf}
\usepackage{pdflscape}
\usepackage{multicol}
\usepackage[margin=0.5in]{geometry}

\usepackage{etoolbox} 
\let\bbordermatrix\bordermatrix 
\patchcmd{\bbordermatrix}{8.75}{4.75}{}{}
\patchcmd{\bbordermatrix}{\left(}{\left[}{}{}
\patchcmd{\bbordermatrix}{\right)}{\right]}{}{}

\usepackage{blkarray}
\usepackage{multirow}

\theoremstyle{plain}
\newtheorem{thm}{Theorem}

\newtheorem{prop}{Proposition}
\newtheorem{lem}{Lemma}
\newtheorem{cor}{Corollary}

\def\Z{\mathbb{Z}}
\def\e{\varepsilon}

\def\a{\alpha}
\def\b{\beta}
\def\s{\sigma}

\theoremstyle{definition}
\newtheorem{example}{Example}

\theoremstyle{remark}

\newtheorem*{rem}{Remark}

\begin{document}

\title{Alexander and writhe polynomials for virtual knots}

\author{Blake Mellor}
\address{Department of Mathematics, Loyola Marymount University}
\email{blake.mellor@lmu.edu}

\begin{abstract}
We give a new interpretation of the Alexander polynomial $\Delta_0$ for virtual knots due to Sawollek \cite{sa} and Silver and Williams \cite{sw}, and use it to show that, for any virtual knot, $\Delta_0$ determines the writhe polynomial of Cheng and Gao \cite{cg} (equivalently, Kauffman's affine index polynomial \cite{ka3}).  We also use it to define a second-order writhe polynomial, and give some applications.
\end{abstract}

\date{\today}

\maketitle

\section{Introduction}

Virtual knots were introduced by Kauffman \cite{ka} as a generalization of classical knot theory, and since then many invariants have been developed to help distinguish virtual knots, and to determine when a virtual knot is equivalent to a classical knot.  In the past few years, several authors have developed invariants that generalize the classical {\em writhe} of a knot \cite{ch, cg, fk, he, ka2, ka3, st}.  These invariants have been used to define Vassiliev invariants of virtual knots \cite{he, ka3}, give bounds on the unknotting number (when it exists) and forbidden number of virtual knots \cite{st, cgm}, and distinguish mutant virtual knots \cite{fk}, among other applications.  These invariants can be unified in a single polynomial invariant, variously called the writhe polynomial \cite{cg}, the affine index polynomial \cite{ka3} and the wriggle polynomial \cite{fk}.  We will refer to it as the {\em writhe polynomial}.

Other authors have extended the classical Alexander polynomials to virtual knots \cite{bo, sa, sw}.  As for knots, there is a sequence of Alexander polynomials $\Delta_k$ for virtual knots, with each polynomial defined modulo the lower-order polynomials.  For classical knots the lowest-order invariant $\Delta_0$ is always trivial, so the most interesting polynomial in the sequence is $\Delta_1$; for virtual knots, however, $\Delta_0$ is generally not trivial, and it has many applications.  In this paper we will provide a new way to look at the polynomial $\Delta_0$, and use this new interpretation to show that the writhe polynomial can be obtained from the Alexander polynomial.

In Section \ref{S:virtual} we will review the definitions of virtual knots and Gauss diagrams, and define the index of a crossing.  In Section \ref{S:alexander} we define the Alexander polynomial $\Delta_0(K)(u,v)$ and give a new interpretation using indices.  In Section \ref{S:writhe} we define the writhe polynomial $W_K(t)$, and show how it is determined by the Alexander polynomial.  In Section \ref{S:2writhe} we extend these ideas to define a ``second-order" writhe polynomial $V_K(t)$.  Finally, in Section \ref{S:applications} we present some applications of our results.  We end with an appendix listing the values of $W_K(t)$ and $V_K(t)$ for all virtual knots with at most 4 crossings.

\section{Virtual knots} \label{S:virtual}

Our approach to virtual knots will be combinatorial. Kauffman \cite{ka} showed that virtual knots can be defined as equivalence classes of diagrams modulo certain moves, generalizing the Reidemeister moves of classical knot theory. Diagrams for virtual knots contain both classical crossings (positive and/or negative crossings, if the knot is oriented) and \emph{virtual} crossings, as shown in Figure \ref{F:crossings}.  Two diagrams are equivalent if they are related by a sequence of the Reidemeister moves shown in Figure \ref{F:reidemeister}. Note that moves (I)--(III) are the classical Reidemeister moves. Kauffman \cite{ka} showed that classical knots are equivalent by this expanded set of Reidemeister moves if and only if they are equivalent by the classical Reidemeister moves, so classical knot theory embeds inside virtual knot theory.

\begin{figure}[htbp]
\begin{center}
\scalebox{.6}{\includegraphics{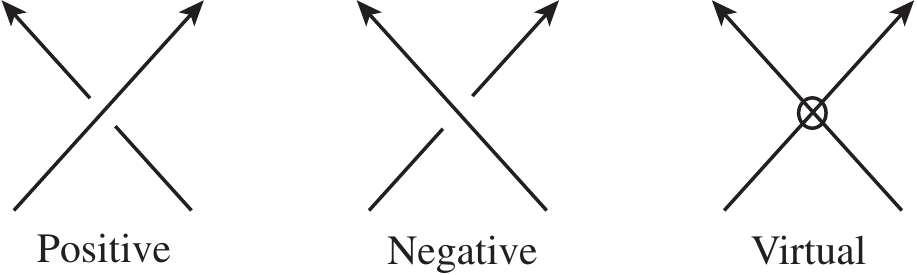}}
\end{center}
\caption{Classical and virtual crossings}
\label{F:crossings}
\end{figure}

\begin{figure}[htbp]
\begin{center}
\scalebox{.8}{\includegraphics{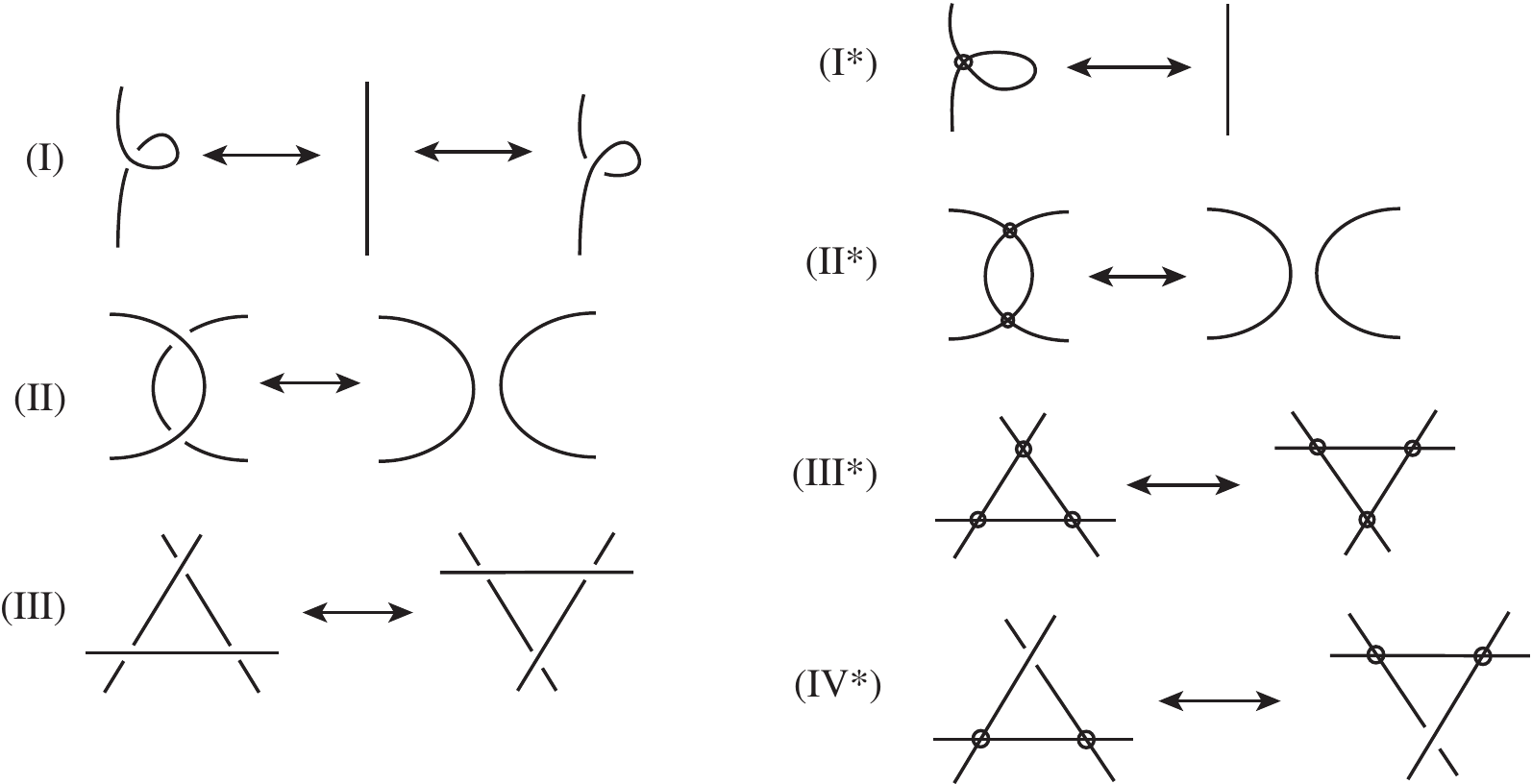}}
\end{center}
\caption{Reidemeister moves for virtual knots}
\label{F:reidemeister}
\end{figure}

\begin{figure}[htbp]
\begin{center}
\scalebox{.6}{\includegraphics{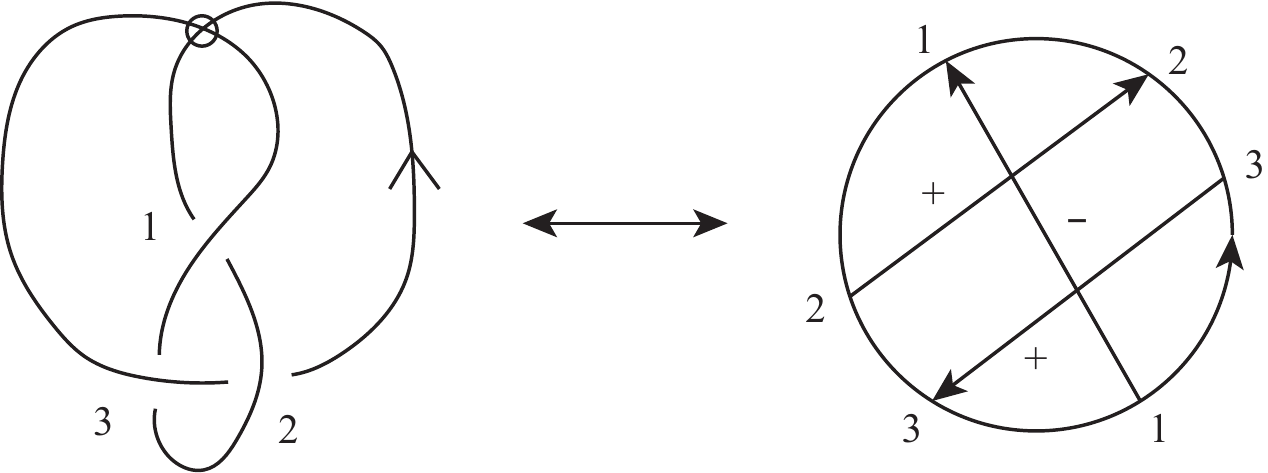}}
\end{center}
\caption{Knot diagram and Gauss diagram for knot with Gauss code U1-O2+U3+O1-O3+U2+.}
\label{F:gaussex}
\end{figure}

One motivation for virtual knots comes from {\em Gauss codes} and {\em Gauss diagrams}.  Any oriented classical knot can be represented by its Gauss code.  We first choose a labeling of the crossings, from 1 to $m$.  The Gauss code is found by selecting a point on the knot, and then making a complete circuit along the knot until we return to this point.  Whenever we pass a crossing, we write down a triple indicating whether it is an over or under-crossing (O or U), the number of the crossing, and the sign of the crossing (so each crossing appears twice in the code, once as an over-crossing and once as an under-crossing).  We can represent this code visually using a Gauss diagram.  The Gauss diagram is an oriented circle with $2m$ points marked along the boundary, corresponding to the $2m$ triples in the Gauss code.  We connect the pair of points corresponding to each crossing with a chord directed from the over-crossing to the under-crossing, and label the chord with the sign of the crossing.  For example, Figure \ref{F:gaussex} shows the knot diagram and Gauss diagram for the knot with Gauss code U1-O2+U3+O1-O3+U2+.

However, it is easy to write down plausible Gauss codes (i.e. each crossing appears in two triples, with the same sign, once with O and once with U) which do not come from classical knots; these codes, and the associated Gauss Diagrams, correspond to virtual knots (where the virtual crossings are ignored in the Gauss code).  Kauffman \cite{ka} showed that the correspondence between Gauss diagrams (modulo equivalents of the classical Reidemeister moves) and virtual knots (modulo the classical and virtual Reidemeister moves) is a bijection.

We will assign several {\it indices} to the chords of a Gauss diagram.  First, we label the endpoint of each chord with a sign.  Let $c = \overrightarrow{PQ}$ be a chord in Gauss diagram $G$, oriented from $P$ to $Q$, with sign $\e(c)$.  We label point $P$ (the over-crossing) with sign $-\e(c)$ and point $Q$ (the under-crossing) with $\e(c)$.  Now let $\a$ be the arc of the bounding circle from $P$ to $Q$, and $\b$ be the arc of the bounding circle from $Q$ to $P$ (both following the orientation of the circle).  So $\a$ is the part of the bounding circle to the right of $c$, and $\b$ is the part of the bounding circle to the left of $c$. \begin{itemize}
	\item The {\em right over-index} of $c$, denoted $RO(c)$, is the sum of the signs of the over-crossing points on arc $\a$.
	\item The {\em right under-index} of $c$, denoted $RU(c)$, is the sum of the signs of the under-crossing points on arc $\a$.
	\item The {\em left over-index} of $c$, denoted $LO(c)$, is the sum of the signs of the over-crossing points on arc $\b$.
	\item The {\em left under-index} of $c$, denoted $LU(c)$, is the sum of the signs of the under-crossing points on arc $\b$.
	\item The {\em index} of $c$ is $Ind(c) = RO(c) + RU(c)$, i.e. the sum of all signs on arc $\b$.
\end{itemize}
Note that $Ind(c) = -LO(c) - LU(c)$, since the sum of all signs around the Gauss diagram is 0.  Figure \ref{F:index} shows a labeled Gauss diagram and its indices.

\begin{figure}[htbp]
\begin{minipage}[b]{.40\textwidth}
\begin{center}
\scalebox{.8}{\includegraphics{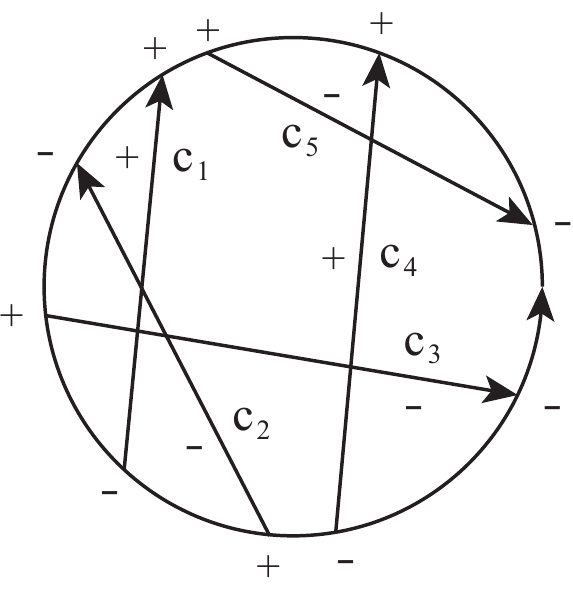}}
\end{center}
\end{minipage} \qquad
\begin{minipage}[b]{.40\textwidth}
\begin{center}
\begin{tabular}[b]{|c|c|c|c|c|c|} \hline
{\em Chord} & $RO$ & $RU$ & $LO$ & $LU$ & $Ind$ \\ \hline
$c_1$ & 1 & -1 & 1 & -1 & 0 \\ \hline
$c_2$ & 0 & 0 & 0 & 0 & 0 \\ \hline
$c_3$ & -1 & 0 & 1 & 0 & -1 \\ \hline
$c_4$ & 0 & -2 & 2 & 0 & -2 \\ \hline
$c_5$ & 0 & -1 & 0 & 1 & -1 \\ \hline
\end{tabular}
\vspace{.5in}
\end{center}
\end{minipage}
\caption{Indices of a Gauss diagram.}
\label{F:index}
\end{figure}

\section{Alexander polynomial} \label{S:alexander}

The Alexander polynomial was extended to virtual knots by Sawollek \cite{sa} and then, using a different approach, by Silver and Williams \cite{sw}.  Silver and Williams proved that Sawollek's polynomial was equivalent to their first Alexander polynomial, $\Delta_0$.  We will use Silver and Williams' polynomial, but we will use a definition presented in \cite[Prop. 4.1]{sw} that incorporates Sawollek's approach.  Given a virtual knot diagram $D$ with $n$ classical crossings, labeled from $c_1$ to $c_n$, an \emph{arc} of the diagram extends from one classical crossing to the next classical crossing (ignoring any virtual crossings).  Note that these go from crossing to crossing, \emph{not} undercrossing to undercrossing (which is the usual notion of an arc in a classical knot diagram).  So $D$ has $2n$ arcs, which we label from $a_1$ to $a_{2n}$.  We choose the labels so that the arcs coming into crossing $c_i$ are labeled $a_{2i-1}$ and $a_{2i}$, as shown in Figure \ref{F:alexander}.  For each $c_i$, we define a $2 \times 2$ matrix $M_i$ to be either $M_+$ or $M_-$, as shown in Figure \ref{F:alexander}, depending on the sign of the crossing.

\begin{figure}[htbp]
\begin{center}
\scalebox{.8}{\includegraphics{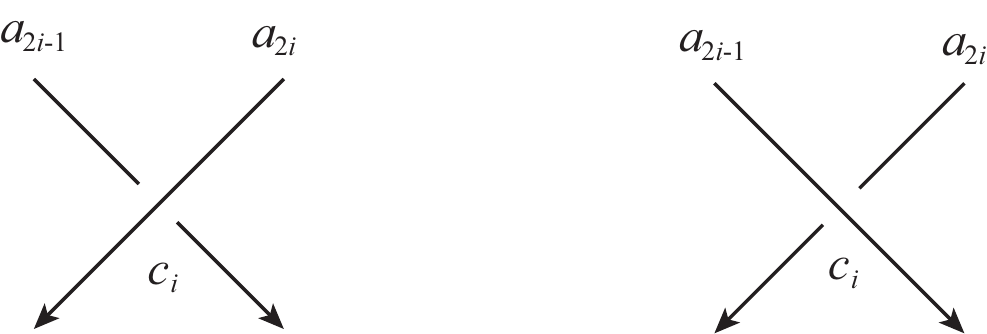}}
\medskip
$$M_+ = \left( \begin{matrix} u^{-1} & 1-(uv)^{-1} \\ 0 & v^{-1} \end{matrix} \right) \qquad\qquad M_- = \left( \begin{matrix} v & 0 \\ 1-uv & u \end{matrix} \right)$$
\end{center}
\caption{A positive (left) and negative (right) crossing}
\label{F:alexander}
\end{figure}

We then let $M$ be the $2n \times 2n$ block diagonal matrix with blocks $M_1, \dots, M_n$.  The rows and columns of $M$ correspond to arcs $a_1, \dots, a_{2n}$, in order.  We also let $P$ be the matrix for the permutation $\pi$ of the arcs of $D$ that is the cycle we read as we go around the knot.  So the entry of $P$ in the $i$th row and $j$th column is 1 if $\pi(a_i) = a_j$ and 0 otherwise.  Then we define
$$\Delta_0(D)(u,v) = \det(M - P).$$

\begin{example}
As an example, we again consider the virtual knot from Figure \ref{F:gaussex}.  In Figure \ref{F:alexander2} we label the arcs of the graph.  The permutation $\pi$ of the arcs induced by the orientation is the cycle $\pi = (a_1a_6a_3a_2a_4a_5)$.

\begin{figure}[htbp]
\begin{center}
\scalebox{.8}{\includegraphics{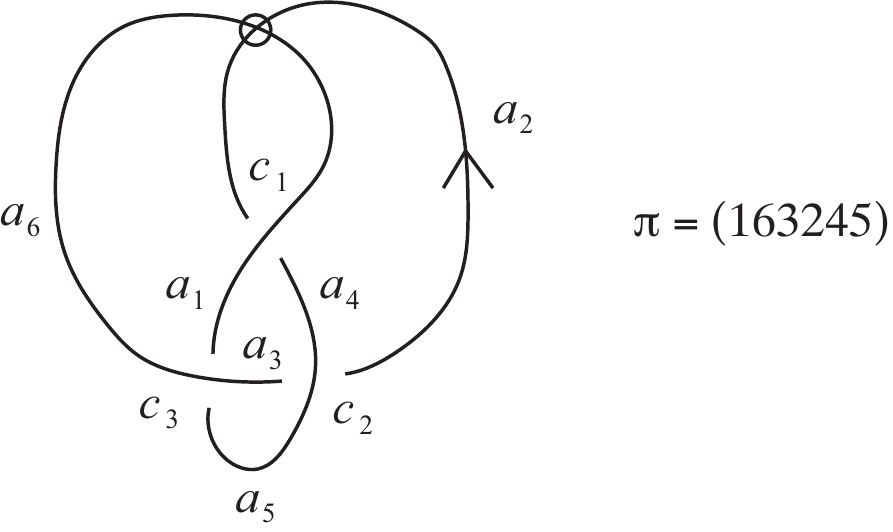}}
\end{center}
\caption{Virtual knot with arc labels, and the permutation $\pi$ induced by the orientation of the knot.}
\label{F:alexander2}
\end{figure}

Now we can write down the matrices $M$ and $P$, and compute the Alexander polynomial.

$$M = \left({\begin{matrix} M_- & 0 & 0 \\ 0 & M_+ & 0 \\ 0 & 0 & M_+ \end{matrix}}\right) = \left({\begin{matrix} v & 0 & 0 & 0 & 0 & 0 \\ 1-uv & u & 0 & 0 & 0 & 0 \\ 0 & 0 & u^{-1} & 1-(uv)^{-1} & 0 & 0 \\ 0 & 0 & 0 & v^{-1} & 0 & 0 \\ 0 & 0 & 0 & 0 & u^{-1} & 1-(uv)^{-1} \\ 0 & 0 & 0 & 0 & 0 & v^{-1} \end{matrix}}\right)$$
$$P = \left({\begin{matrix} 0&0&0&0&0&1 \\ 0&0&0&1&0&0 \\ 0&1&0&0&0&0 \\ 0&0&0&0&1&0 \\ 1&0&0&0&0&0 \\ 0&0&1&0&0&0 \end{matrix}}\right)$$

$$\Delta_0(D)(u,v) = \det(M-P) = (1-u)(1-v)(1-uv)u^{-1}v^{-1}$$
\end{example}

To turn $\Delta_0$ from an invariant of diagrams into an invariant of virtual knots, we look at how it changes under the virtual Reidemeister moves.  The virtual moves have no effect on the invariant, so we only need to consider the classical moves.  When we add orientations to the moves shown in Figure \ref{F:reidemeister}, there are quite a few cases, but Polyak \cite{po} showed these are all generated by the four oriented Reidemeister moves $I_a, I_b, II_a, III_a$ shown in Figure \ref{F:genreid}.

\begin{figure}[htbp]
\begin{center}
\scalebox{.8}{\includegraphics{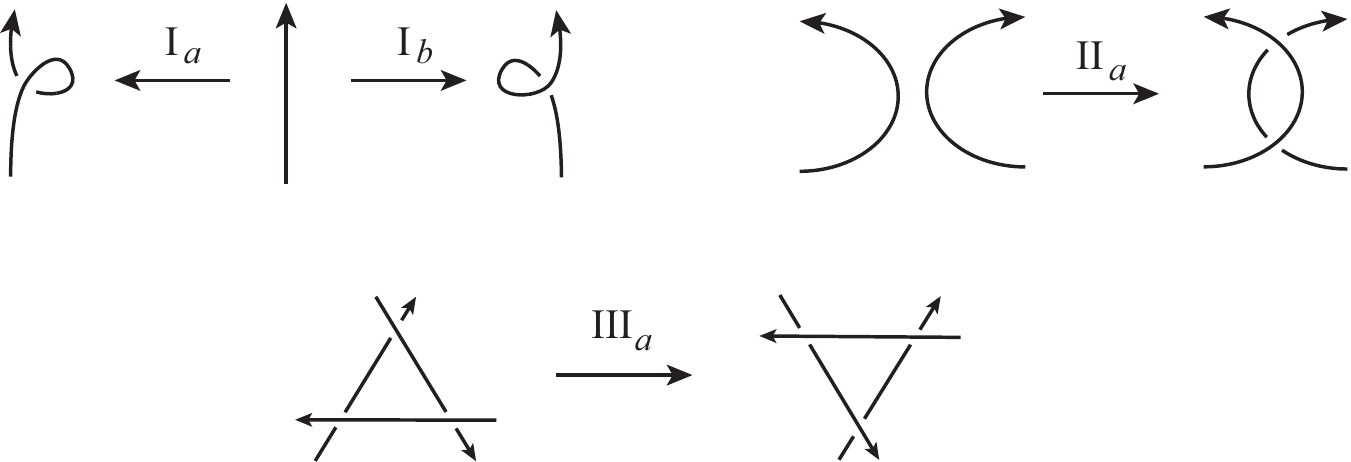}}
\end{center}
\caption{Generating set for oriented Reidemeister moves}
\label{F:genreid}
\end{figure}

The following proposition is a straightforward exercise in analyzing $\det(M-P)$ before and after each of the Reidemeister moves.

\begin{prop}\label{P:alexreid}
Let $D$ be a virtual knot diagram, and $D'$ the result of applying one of the moves $I_a, I_b, II_a, III_a$.  Then \begin{itemize}
	\item[$I_a$:] $\Delta_0(D')(u,v) = \Delta_0(D)(u,v)$
	\item[$I_b$:] $\Delta_0(D')(u,v) = (uv)^{-1}\Delta_0(D)(u,v)$
	\item[$II_a$:] $\Delta_0(D')(u,v) = \Delta_0(D)(u,v)$
	\item[$III_a$:] $\Delta_0(D')(u,v) = \Delta_0(D)(u,v)$
\end{itemize}
\end{prop}

So $\Delta_0$ is well-defined for virtual knots, modulo multiplication by powers of $uv$.  So we can normalize the polynomial by multiplying by $(uv)^k$, where $k$ is the minimum among all powers of $u$ (this matches the normalization used by Silver and Williams \cite{sw}).  We will use $\Delta_0(K)$ to refer to this normalized polynomial, which is now an invariant of virtual knots.

Silver and Williams prove several properties of $\Delta_0(K)$, including the following:

\begin{prop}\label{P:factor}\cite{sw}
Let $K$ be an oriented virtual knot.  Then $(1-u)(1-v)(1-uv)$ divides $\Delta_0(K)(u,v)$.
\end{prop} \medskip


\begin{prop}\label{P:reverse} \cite{sw}
Given a diagram $D$ of a virtual knot $K$, let $D^\#$ be the result of switching every (classical) crossing of $D$, $D^*$ be the reflection across a vertical line in the plane of the diagram, and $-D$ the result of reversing all orientations.  Let $K^\#$, $K^*$ and $-K$ be the corresponding virtual links.  Then for all $i \geq 0$,
\begin{enumerate}
	\item $\Delta_i(K^\#)(u,v) = -\Delta_i(K)(v,u)$
	\item $\Delta_i(K^*)(u,v) = \Delta_i(K)(u^{-1}, v^{-1})$
	\item $\Delta_i(-K)(u,v) = -\Delta_i(K)(u^{-1}, v^{-1})$
\end{enumerate}
\end{prop}

\subsection{Indices and the Alexander polynomial} \label{SS:indices}

Let $D$ be a diagram for a virtual knot $K$.  If we denote the entries of the matrix $M-P$ by $b_{ij}$, then 
$$\Delta_0(D)(u,v) = \det(M-P) = \sum_{\s \in S_{2n}}{(-1)^\s b_{1,\s(1)}\cdots b_{2n, \s(2n)}}$$

In this section, we will show how the terms of this sum can be interpreted as counting crossings in links associated to the diagram $D$.  First, we need to define some terminology.  Let $\pi$ be the permutation on $\{1, \dots, 2n\}$ given by the cycle read around the diagram $D$ (so if $a_i$ is an arc into a crossing, $a_{\pi(i)}$ is the corresponding arc leaving the crossing).  Then we have: \begin{itemize}
	\item For each $i$, $b_{i, \pi(i)} = -1$.
	\item If $c_i$ is a positive crossing, then $b_{2i-1, 2i-1} = u^{-1}$, $b_{2i, 2i} = v^{-1}$ and $b_{2i-1, 2i} = 1-(uv)^{-1}$.
	\item If $c_i$ is a negative crossing, then $b_{2i-1, 2i-1} = v$, $b_{2i, 2i} = u$ and $b_{2i, 2i-1} = 1-uv$.
	\item All other $b_{i,j} = 0$.
\end{itemize}

The two exceptions to this are when $c_i$ is a positive crossing and $\pi(2i-1) = 2i$ or when $c_i$ is negative and $\pi(2i) = 2i-1$ (we will refer to these as {\em special curls}).  In the first case we have $b_{2i-1, 2i} = (1-(uv)^{-1}) - 1 = -(uv)^{-1}$ and in the second we have $b_{2i, 2i-1} = (1-uv) - 1 = -uv$.

Let $C$ be a subset of the (classical) crossings of $D$.  A permutation $\s \in S_{2n}$ {\em corresponds} to $C$ if $\s(2i-1) = 2i$ for any positive crossing $c_i$ in $C$, $\s(2i) = 2i-1$ for any negative crossing $c_i$ in $C$, and $\s(t) = t$ or $\pi(t)$ for all other $t \in \{1, \dots, 2n\}$.  Notice that if a permutation does not correspond to some set of crossings, then $b_{t, \s(t)} = 0$ for some $t$, and the corresponding term in $\det(M-P)$ is trivial.  So when computing $\Delta_0(D)(u,v)$, we are only interested in permutations which correspond to some set of crossings.

If $C$ has a special curl $c_i$ and $\s$ is a permutation corresponding to $C$, we will find it convenient to split the term associated to $\s$ in $\det(M-P)$ into two terms.  If $c_i$ is positive, then one term has $b_{2i-1, \s(2i-1)} = b_{2i-1, 2i} = 1-(uv)^{-1}$ and the other term has $b_{2i-1, \s(2i-1)} = b_{2i-1, \pi(2i-1)} = -1$ (and similarly if $c_i$ is negative).  We then consider the first term as coming from the permutation $\s$ corresponding to $C$, and the second term as coming from the permutation $\s$ corresponding to $C' = C - \{c_i\}$.  With this convention, we do not need to consider special curls separately in the proofs in this section.

A set of crossings $C$ is {\em alternating} in $D$ if, as we go around the knot, we alternate between overcrossings and undercrossings in $C$ (ignoring the other crossings).  The corresponding chords in the Gauss diagram for $D$ are called an {\em alternating configuration}; their endpoints will alternate between overcrossings and undercrossings as we go around the boundary circle.  There are exactly two alternating configurations of 2 chords, shown in Figure \ref{F:configuration}.  There are five alternating configurations of 3 chords.  In general, there are fewer than $n!$ alternating configurations of $n$ chords, found by alternating $n$ $O$'s and $n$ $U$'s around a circle and matching them.  The exact count is somewhat complicated, because many configurations may be equivalent by a rotation of the diagram.  Our next lemma shows that {\em only} permutations corresponding to alternating sets of crossings contribute nonzero terms to $\Delta_0(D)$.

\begin{figure}[htbp]
\begin{center}
\scalebox{.8}{\includegraphics{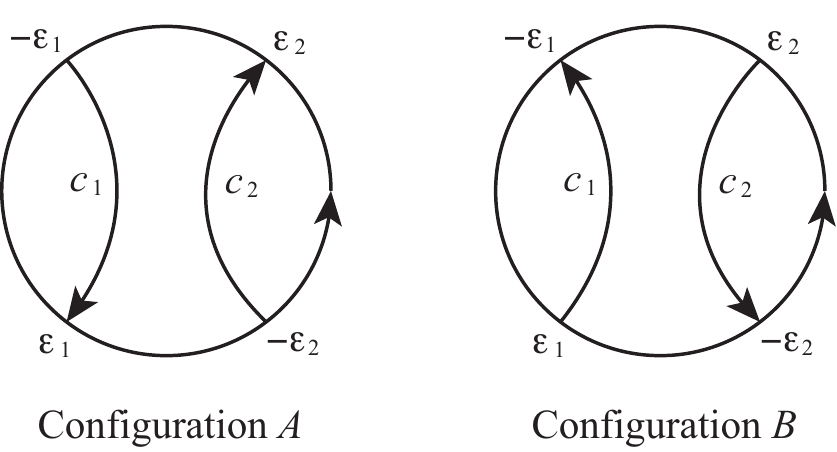}}
\end{center}
\caption{Alternating configurations of pairs of chords (i.e. crossings).}
\label{F:configuration}
\end{figure}

\begin{lem} \label{L:nonalternating}
Let $C$ be a set of crossings in a virtual knot diagram $D$.  If $C$ is not an alternating configuration, then the contribution of a permutation corresponding to $C$ to $\Delta_0(D)(u,v)$ is 0.
\end{lem}
\begin{proof}
Since $C$ is not alternating, it contains at least two consecutive overcrossings (and at least two consecutive undercrossings).  So there are crossings $c_i$ and $c_j$ in $C$ such that there is a segment of $D$ running from the overcrossing of $c_i$ to the overcrossing of $c_j$ without passing through any other crossings of $C$.  We will assume that $c_i$ and $c_j$ are positive crossings (the other cases are similar).

A permutation $\s$ corresponding to $C$ must have $\s(2i-1) = 2i$ and $\s(2j-1)=2j$, where arcs $a_{2i}$ and $a_{2j}$ are the overcrossing arcs.  Hence, $\s(2i) \neq 2i$, so the only other possibility where $b_{2i, \s(2i)} \neq 0$ is $\s(2i) = \pi(2i)$.  Then $\s^2(2i) = \s(\pi(2i)) \neq \pi(2i)$ (since $2i \neq \pi(2i)$), so $\s^2(2i) = \pi^2(2i)$ (otherwise the term is 0).  Continuing in this way, $\s^k(2i) = \pi^k(2i)$, until $\s^r(2i) = \pi^r(2i) = 2j$.  But then $\s^{r-1}(2i) = \pi^{r-1}(2i) = \s^{-1}(2j) = 2j-1$.  But then the segment from the overcrossing of $c_i$ goes first to the {\em undercrossing} of $c_j$, which is a contradiction.  So there must be some $b_{\s^k(2i), \s^{k+1}(2i)}$ which is 0, and therefore the term is 0.
\end{proof}

An alternating configuration $C$ of $k$ crossings in a virtual knot divides the knot into $2k$ segments.  Of these segments, $k$ run from an undercrossing in $C$ to a subsequent overcrossing in $C$ (possibly going through other crossings that are not in $C$); we will call these segments {\em ascending}.  The other $k$, called {\em descending} segments, run from overcrossings in $C$ to undercrossings in $C$.  A {\em smoothing} of a configuration $C$ is the result of smoothing every crossing in $C$ as shown in Figure \ref{F:smoothing}.  Note that the smoothing joins the two ascending segments at a crossing and the two descending segments.  

\begin{figure}[htbp]
\begin{center}
\scalebox{1}{\includegraphics{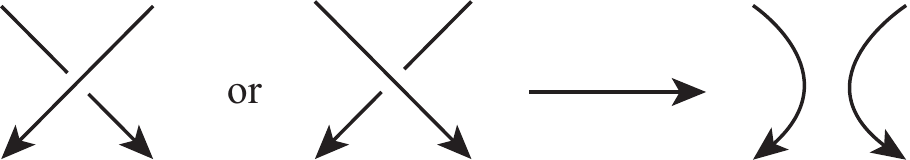}}
\end{center}
\caption{Smoothing a crossing in a configuration $C$.}
\label{F:smoothing}
\end{figure}

After we've smoothed the configuration $C$, we are left with a link where each component is composed entirely of ascending or descending segments from the original knot; we refer to the components as ascending or descending accordingly.

\begin{prop} \label{P:alexdet3}
Let $C$ be an alternating configuration of $m$ crossings in a virtual knot diagram $D$ (where $m \geq 1$).  Let $L = L_d \cup L_a$ be the link formed by smoothing the crossings in $C$, where $L_d$ is the sublink formed from the descending components of $L$, and $L_a$ is the sublink formed from ascending components of $L$.  Then at most one permutation corresponding to $C$ makes a non-zero contribution to $\Delta_0(D)$, and its contribution is:
$$(1-uv)^{neg(C)}(1-(uv)^{-1})^{pos(C)}(-1)^{m+\vert L_d \vert}u^{-U(L_a)}v^{-O(L_a)}$$
where $U(L_a)$ and $O(L_a)$ are the number of under- and over-crossings along the components of $L_a$, counted with sign, $neg(C)$ and $pos(C)$ are the number of negative and positive crossings in $C$, and $\vert L_d \vert$ is the number of components of $L_d$.
\end{prop}
\begin{proof}
Let $\s$ be a permutation corresponding to $C$.  In other words, $b_{2i, \s(2i)} = 1-uv$ or $b_{2i-1, \s(2i-1)} = 1-(uv)^{-1}$ if and only if $c_i$ is in $C$.  Suppose that, for all $t$, $b_{t, \s(t)} \neq 0$.

Let $c_i$ and $c_j$ be two adjacent crossings in $C$ (so there is a segment of $D$ from $c_i$ to $c_j$ that does not pass through any other crossings of $C$).  Assume that $c_i$ and $c_j$ are both positive (the other cases are similar).  Then $\s(2i-1) = 2i$ and $\s(2j-1) = 2j$.  

We will first suppose that the descending segment of $D$ running from the overcrossing of $c_i$ to the undercrossing of $c_j$ does not pass through any other crossings in $C$.  So for some $r$, $\pi^r(2i) = 2j-1$.  Since $b_{2i, \s(2i)} \neq 0$, we must have $\s(2i) = 2i$ or $\s(2i) = \pi(2i)$.  But $\s(2i-1) = 2i$, so the only possibility is $\s(2i) = \pi(2i) \neq 2i$.  But then, $\s^2(2i) \neq \s(2i)$, so $\s^2(2i) = \pi(\s(2i)) = \pi^2(2i)$.  Continuing in this way, $\s^k(2i) = \pi^k(2i)$ for $1 \leq k \leq r$.  If $a_t$ is on this segment, so $t = \pi^{k-1}(2i)$ for some $1 \leq k \leq r$, then $b_{t, \s(t)} = b_{t, \pi(t)} = -1$.

Now suppose instead that the ascending segment of $D$ running from the undercrossing of $c_i$ to the overcrossing of $c_j$ does not pass through any other crossings in $C$.  So for some $r$, $\pi^r(2i-1) = 2j$.  Since $\pi^{r-1}(2i-1) \neq 2j-1$, and $b_{\pi^{r-1}(2i-1), \s(\pi^{r-1}(2i-1))} \neq 0$, we must have $\s(\pi^{r-1}(2i-1)) = \pi^{r-1}(2i-1)$.  Continuing in this way, $\s(\pi^{r-k}(2i-1)) = \pi^{r-k}(2i-1)$ for $1 \leq k \leq r-1$.  If $a_t$ is on this segment, so $t = \pi^k(2i-1)$ for some $1 \leq k \leq r-1$, then $b_{t, \s(t)} = b_{t,t} = u^{\pm 1}$ or $v^{\pm 1}$, depending on the crossing (as shown in Figure \ref{F:contribution}).

\begin{figure}[htbp]
\begin{center}
\scalebox{.8}{\includegraphics{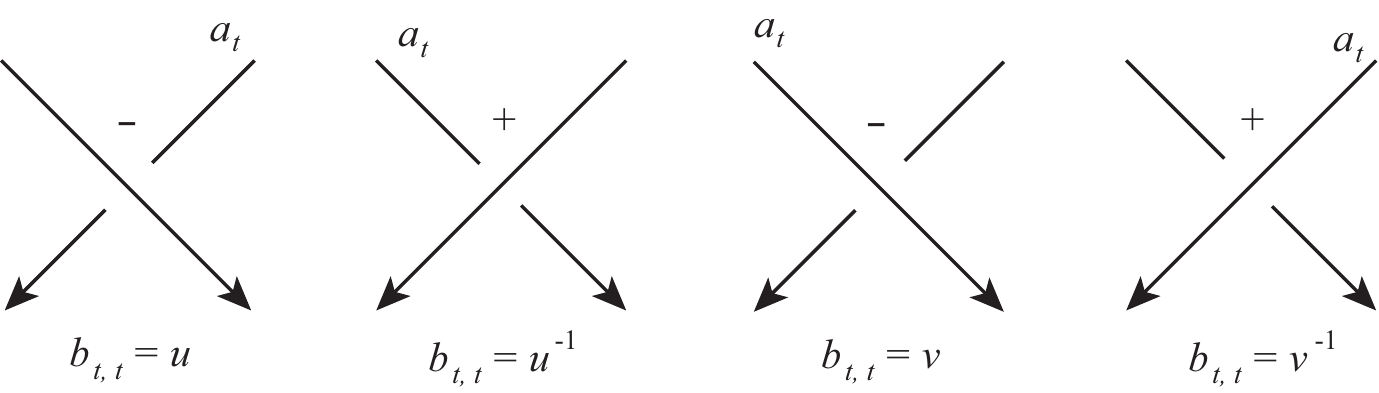}}
\end{center}
\caption{Values of $b_{t,t}$ at each crossing along $\a$.}
\label{F:contribution}
\end{figure}

So there is a unique permutation $\s$ corresponding to $C$ with $b_{t, \s(t)} \neq 0$ for all $t$.  $\s$ has a one-cycle for each arc on a component of $L_a$.  $\s$ also has a cycle for each component of $L_d$ of length $a+s$, where $a$ is the number of arcs in the component and $s$ is the number of smoothings.  This cycle has the form $(\dots, 2i-1, 2i, \pi(2i), \pi^2(2i), \dots, \pi^{r-1}(2i), 2j-1, 2j, \dots)$.  So the sign of the permutation $\s$ is the product over all components of $L_d$ of $(-1)^{a+s+1}$.  Then the contribution to the determinant is:
\begin{align*}
(-1)^{\s}b_{1,\s(1)}\cdots b_{2n, \s(2n)} &= (-1)^{\sum_{L_d}{(a+s+1)}}(1-uv)^{neg(C)}(1-(uv)^{-1})^{pos(C)}(-1)^{\sum_{L_d}{a}}u^{-U(L_a)}v^{-O(L_a)} \\
&= (1-uv)^{neg(C)}(1-(uv)^{-1})^{pos(C)}(-1)^{\sum_{L_d}{s+1}}u^{-U(L_a)}v^{-O(L_a)} \\
&= (1-uv)^{neg(C)}(1-(uv)^{-1})^{pos(C)}(-1)^{m+ \vert L_d\vert}u^{-U(L_a)}v^{-O(L_a)}
\end{align*}

It is possible that a segment of $D$ will run from a crossing $c_i$ back to $c_i$.  However, the arguments above work equally well if $j = i$.
\end{proof}

\begin{figure}[htbp]
\begin{center}
\scalebox{.8}{\includegraphics{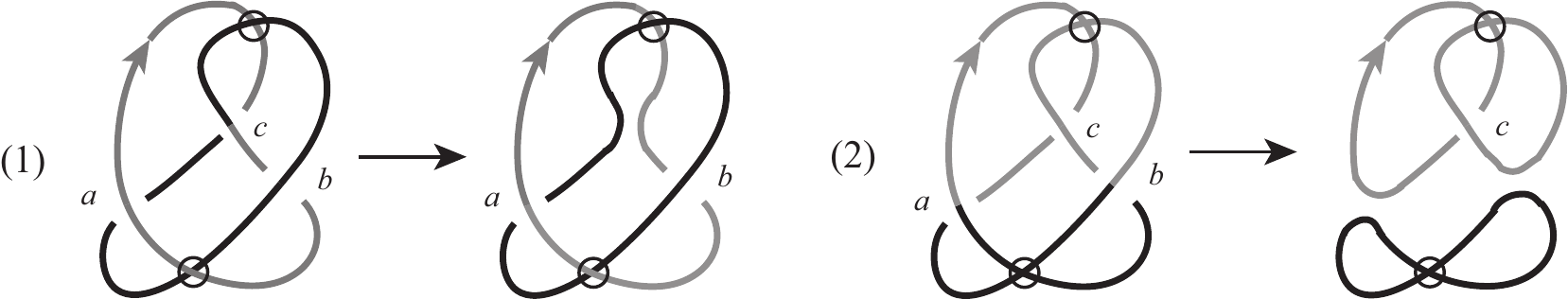}}
\end{center}
\caption{Contributions of alternating configurations to $\Delta_0(D)$.}
\label{F:Prop4Example}
\end{figure}

\begin{example}
As an example, consider the virtual knot 3.1 in Green's table \cite{gr}, shown in Figure \ref{F:Prop4Example}.  This knot has three crossings, $a$, $b$ and $c$; $a$ is a positive crossing while $b$ and $c$ are negative crossings.  Each crossing by itself is an alternating configuration, and $\{a, b\}$ is an alternating configuration of two crossings.  We will use Proposition \ref{P:alexdet3} to compute the contributions from (1) configuration $\{c\}$ and (2) configuration $\{a,b\}$.
\begin{enumerate}

\item The result of smoothing crossing $c$ is shown on the left in Figure \ref{F:Prop4Example}.  $L_d$ is shown in gray.  Observe that $\vert L_d \vert = 1$, $U(L_a) = 1$ (since $a$ is a positive crossing) and $O(L_a) = -1$ (since $b$ is a negative crossing).  So the contribution of this configuration to $\Delta_0(D)$ is:
$$(1-uv)(-1)^{1+1}u^{-1}v^1 = u^{-1}v - v^2$$

\item The result of smoothing configuration $\{a, b\}$ is shown on the right in Figure \ref{F:Prop4Example}. Again, $L_d$ is shown in gray.  Observe that $\vert L_d \vert = 2$ and $U(L_a) = O(L_a) = 0$ (since $L_a$ has no crossings).  So the contribution of this configuration to $\Delta_0(D)$ is:
$$(1-uv)(1-(uv)^{-1})(-1)^{2+2}u^{0}v^0 = 2-uv-u^{-1}v^{-1}$$

\end{enumerate}
\end{example}
\medskip

We can group the terms in the sum for $\det(M-P)$ according to how many factors of $(1-uv)$ or $(1-(uv)^{-1})$ they contain (after splitting any special curls as described previously).  Note that $1- (uv)^{-1} = -(uv)^{-1}(1-uv)$.  Then we can write
$$\Delta_0(D)(u,v) = \sum_{k=0}^{n}{(1-uv)^k f_k(u,v)}$$

We will give precise descriptions of $f_0, f_1$ and $f_2$.  Recall that the {\em writhe} of a knot diagram $D$, denoted $Wr(D)$,  is the sum of the signs of all crossings in the diagram.

\begin{prop} \label{P:alexdet1}
Given a virtual knot diagram $D$ of a virtual knot, then \begin{enumerate}
	\item $f_0(u,v) = (uv)^{-Wr(D)} - 1$, and
	\item ${\displaystyle f_1(u,v) = \sum_c{-\e(c)(uv)^{-(1+\e(c))/2}u^{-LU(c)}v^{LO(c)}}}$, where the sum is over all classical crossings $c$.
	\item ${\displaystyle f_2(u,v) = \sum_{\{c_1, c_2\}\in S_1}{\e_1\e_2(uv)^{-(\e_1+1)/2 - (\e_2+1)/2}u^{-LU(c_1) + RU(c_2) + \e_2}v^{LO(c_1) - RO(c_2) + \e_2}}}$

$$ - \sum_{\{c_1, c_2\}\in S_2}{\e_1\e_2(uv)^{-(\e_1+1)/2 - (\e_2+1)/2}u^{-LU(c_1) - LU(c_2)}v^{LO(c_1) + LO(c_2)}}$$
where $S_1$ is the set of all (unordered) pairs of crossings in configuration $A$ and $S_2$ is the set of all pairs of crossings in configuration $B$ (in the Gauss diagram), as shown in Figure \ref{F:configuration}.
\end{enumerate}
\end{prop}
\begin{proof}
The proof of the first part was contained in the proof of part (2) of Proposition \ref{P:factor} given in \cite{sw}.  We repeat it here for clarity, using the terminology we've developed.  As before, we can write the determinant of $B = M-P$ as
$$\sum_{\s \in S_{2n}}{(-1)^\s b_{1, \s(1)}\cdots b_{2n, \s(2n)}}$$
To find $f_0$, we assume that we do not use any entries equal to $1-uv$ or $1-(uv)^{-1}$ (i.e. we consider permutations corresponding to the empty set of crossings).  So we only consider $\s$ where $\s(i) = i$ or $\s(i) = \pi(i)$ for every $i$.  Suppose there is some $i$ for which $\s(i) = \pi(i) \neq i$.  But then $\s^2(i) \neq \s(i)$, so we must have $\s^2(i) = \pi(\s(i)) = \pi^{2}(i)$.  Continuing in this way, $\s^k(i) = \pi^{k}(i)$ for all $k$. Since $\pi$ is a single $2n$-cycle, this means $\s = \pi$.  So there are only two permutations we need to consider: $\s = id$ and $\s = \pi$.  Since $\pi$ is a cycle of even length, $(-1)^{\pi} = -1$.  Hence,
\begin{align*}
f_0(u,v) &= (b_{1,1}\cdots b_{2n,2n}) - (b_{1,\pi(1)}\cdots b_{2n,\pi(2n)}) \\
&=(uv)^{\#\ negative\ crossings - \#\ positive\ crossings} - (-1)^{2n} \\
&= (uv)^{-Wr(D)} - 1
\end{align*}

Now we consider $f_1$.  So we are considering sets of crossings with only one element, $C = \{c\}$.  The link created by smoothing this crossing will have two components, one ascending and one descending.  By Proposition \ref{P:alexdet3}, the contribution is $(1-uv)(-1)^{1+1}u^{-U(L_a)}v^{-O(L_a)} = (1-uv)u^{-U(L_a)}v^{-O(L_a)}$ if $c$ is negative, and $(1-(uv)^{-1})u^{-U(L_a)}v^{-O(L_a)}$ if $c$ is positive.  Since $1-(uv)^{-1} = -(uv)^{-1}(1-uv)$, once we factor out $1-uv$ the contribution to $f_1(u,v)$ is $-\e(c)(uv)^{-(1+\e(c))/2}u^{-U(L_a)}v^{-O(L_a)}$.

Since $C$ contains only one crossing, there is only one ascending segment of the knot diagram.  In the Gauss diagram, this corresponds to the arc of the boundary circle to the left of the chord $c$.  Hence $-U(L_a) = -LU(c)$ and $-O(L_a) = LO(c)$ (remember that we label overcrossings with $-\e(c)$, hence the change in sign).  So then the contribution of $c$ to $f_1$ is $-\e(c)(uv)^{-(1+\e(c))/2}u^{-LU(c)}v^{LO(c)}$, as desired.

Finally, we turn to $f_2(u,v)$.  Let $C = \{c_i, c_j\}$ be a set of two crossings; by Lemma \ref{L:nonalternating} we only need to consider when $C$ is alternating.  Then by Proposition \ref{P:alexdet3} (with $m = 2$), the contributions of the permutations corresponding to $C$ to $\det(M-P)$ is 
$$(1-uv)^{neg(C)}(1-(uv)^{-1})^{pos(C)}(-1)^{\vert L_d \vert}u^{-U(L_a)}v^{-O(L_a)}.$$

In configuration $A$, $\vert L_d \vert = 2$.  The two ascending segments of the boundary circle are on the left side of both chords.  Counting the crossings along these arcs is equivalent to counting the difference between those to the left of chord $c_i$ (except chord $c_j$ itself) and those to the right of chord $c_j$.  So the term is
$$(1-uv)^{neg(C)}(1-(uv)^{-1})^{pos(C)} u^{-(LU(c_i) - RU(c_j) - \e(c_j))}v^{LO(c_i) - RO(c_j) + \e(c_j)}$$
$$=(1-uv)^{neg(C)}(1-(uv)^{-1})^{pos(C)} u^{-LU(c_i) + RU(c_j) + \e(c_j)}v^{LO(c_i) - RO(c_j) + \e(c_j)}$$

If the crossings are in configuration $B$, then $\vert L_d \vert = 1$.  One of the two ascending segments is on the left of chord $c_i$, and the other is on the left of chord $c_j$.  So the term is
$$-(1-uv)^{neg(C)}(1-(uv)^{-1})^{pos(C)} u^{-(LU(c_i) + LU(c_j))}v^{LO(c_i) + LO(c_j)}$$
$$=-(1-uv)^{neg(C)}(1-(uv)^{-1})^{pos(C)}u^{-LU(c_i) - LU(c_j)}v^{LO(c_i) + LO(c_j)}$$

As with $f_1$, $(1-uv)^{neg(C)}(1-(uv)^{-1})^{pos(C)} = \e_i\e_j(uv)^{-(\e_i+1)/2 - (\e_j+1)/2}$, completing the proof.
\end{proof}

\begin{rem}
If we switch the roles of $c_1$ and $c_2$ in the first sum in the formula for $f_2(u,v)$, we will get the same exponents, since $LU(c_1) + RU(c_1) + \e_1 = LU(c_2) + RU(c_2) + \e_2 = Wr(D)$, and similarly for the overcrossings.
\end{rem}

A similar result could be found for each $f_k(u,v)$.  In general, we would need to look at each of the alternating configurations of $n$ chords.  However, the formulas will quickly become very complex; even $f_3$ involves five different terms.

\section{Writhe polynomial} \label{S:writhe}

Now we turn our attention to writhes.  We have already mentioned the {\em writhe} of a virtual knot diagram, $Wr(D)$.  This has been generalized by several authors \cite{ka2, ch, cg, st}.  We will use the $n$-writhe defined by Satoh and Taniguchi \cite{st}.  Unlike the writhe, where we sum the signs of all crossings, for the $n$-writhe we sum the signs of crossings with index $n$ (see Section \ref{S:virtual}).  So given a knot diagram (or, equivalently, Gauss diagram) $D$, the {\em $n$-writhe} of $D$ is $w_n(D) = \sum_{Ind(c) = n}{\e(c)}$.

\begin{figure}[htbp]
\begin{center}
\scalebox{.6}{\includegraphics{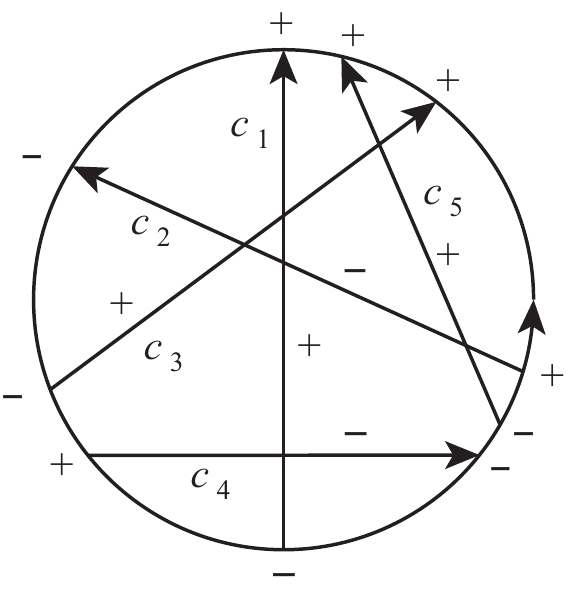}}
\end{center}
\caption{A labeled Gauss diagram $D$}
\label{F:nwrithe}
\end{figure}

\begin{example}
Consider the Gauss diagram $D$ of degree 5 shown in Figure \ref{F:nwrithe}.  The indices of the chords are:
$$Ind(c_1) = 1, \qquad Ind(c_2) = 3, \qquad Ind(c_3) = -1, \qquad Ind(c_4) = -1, \qquad Ind(c_5) = 2$$
and the $n$-writhes are
$$w_{-1}(D) = 0, \qquad w_1(D) = 1, \qquad w_2(D) = 1, \qquad w_3(D) = -1, \qquad w_n(D) = 0 {\rm \ (for\ all\ other\ }n)$$
\end{example}

Satoh and Taniguchi showed that $w_n$ is a virtual knot invariant for $n \neq 0$.  We will extend this result, with a modification, to $n = 0$.

\begin{lem} \label{L:nwrithe}
Let $D$ and $D'$ be virtual knot diagrams related by a finite sequence of Reidemeister moves.  Then  \begin{enumerate}
	\item $w_n(D) = w_n(D')$ for any $n \neq 0$. \cite{st}
	\item $w_0(D) - Wr(D) = w_0(D') - Wr(D')$.
\end{enumerate}
\end{lem}
\begin{proof}
The first part was proved by Satoh and Taniguchi \cite{st}.  Reidemeister moves of type (II) add (or remove) two crossings of opposite sign and the same index, while moves of type (III) preserve the signs and indices of all crossings \cite{st}.  So these moves preserve both $w_n$ and the writhe.  A Reidemeister move of type 1 adds or removes a crossing of index 0 (since all other chords have both endpoints on one side), so both $w_0$ and the writhe are increased or decreased by 1 (all other $w_n$ are unchanged).  Hence the difference $w_0(D) - Wr(D)$ is preserved.
\end{proof}

So for any virtual knot $K$ with diagram $D$, we can define $w_n(K) = w_n(D)$ (for $n \neq 0$) and (slightly abusing notation) $w_0(K) = w_0(D) - Wr(D)$; these are well-defined invariants of the virtual knot.  We can combine these to define a {\em writhe polynomial} $W_K(t)$ by
$$W_K(t) = \sum_{n \in \Z}{w_n(K)t^n}$$
Since at most a finite number of the $w_n$'s are nonzero (in particular, $w_n = 0$ whenever $\vert n \vert$ is larger than the number of crossings), $W_K(t)$ is a well-defined Laurent polynomial.  In a classical knot $K$, all crossings have index 0, so $W_K(t) = 0$.  $W_K(t)$ is equivalent to both Kauffman's Affine Index Polynomial \cite{ka3} and Cheng and Gao's writhe polynomial \cite{cg}.  As remarked by Satoh and Taniguchi \cite{st}, it also determines Henrich's index polynomial \cite{he}.

By Proposition \ref{P:factor}, we can define $\Delta'_0(K)$ by $\Delta_0(K)(u,v) = (1-uv)\Delta'_0(K)(u,v)$. Our main result describes how $W_K(t)$ is determined by $\Delta'_0(K)$.  

\begin{thm} \label{T:alexwrithe}
For any virtual knot $K$, $W_K(t) = -\Delta'_0(K)(t, t^{-1})$.
\end{thm}
\begin{proof}
Let $D$ be a diagram for $K$.  Recall that $\Delta_0(D)(u,v) = f_0(u,v) + (1-uv)f_1(u,v) + (1-uv)^2f_2(u,v) + \cdots + (1-uv)^nf_n(u,v)$.  From Proposition \ref{P:alexdet1}, $f_0(u,v) = (uv)^{-Wr(D)} - 1 = (uv)^{-Wr(D)}(1-(uv)^{Wr(D)}) = (uv)^{-Wr(D)}(1-uv)\sum_{i=0}^{Wr(D)-1}{(uv)^i}$.  So then 
$$\Delta'_0(D)(u,v) = (uv)^{-Wr(D)}\left(\sum_{i=0}^{Wr(D)-1}{(uv)^i}\right) + f_1(u,v) + (1-uv)f_2(u,v) + \cdots + (1-uv)^{n-1}f_n(u,v)$$
When we set $t = u = v^{-1}$, then $uv = 1$, so all terms after $f_1(u,v)$ disappear.  From Proposition \ref{P:alexdet1}, we get
\begin{align*}
\Delta'_0(D)(t,t^{-1}) &= (1)^{-Wr(D)}\left(\sum_{i=0}^{Wr(D)-1}{(1)^i}\right) + \sum_c{-\e(c)t^{-LU(c)}t^{-LO(c)}} \\
&= Wr(D) - \sum_c{\e(c)t^{Ind(c)}} \\
&= -W_K(t)
\end{align*}
\end{proof}

As a corollary, since $\Delta'_0$ has a factor of $(1-u)(1-v)$ by Proposition \ref{P:factor}, the writhe polynomial always has a factor of $(1-t)(1-t^{-1})$.

\begin{cor} \label{C:factor}
For any virtual knot $K$, $(1-t)(1-t^{-1})$ divides $W_K(t)$.  In particular, the sum of the coefficients of $W_K(t)$ is 0.
\end{cor}

This allows us to prove a conjecture of Benioff and the author \cite{bm} that originally motivated this investigation.  The {\em odd writhe} of a virtual knot (introduced by Kauffman \cite{ka2}) is the sum of the $n$-writhes where $n$ is odd.  Let $ow(K)$ denote the odd writhe, and recall that $\Delta_0(K)(u,v) = (1-u)(1-v)(1-uv)\overline{\Delta}_0(K)(u,v)$.

\begin{cor} \label{C:oddwrithe}
For any virtual knot $K$, $ow(K) = 2\overline{\Delta}_0(K)(-1,-1)$.
\end{cor}
\begin{proof}
Recall that $W_K(t) = \sum_{n \in \Z}{w_n(K) t^n} = \sum_{even\ n}{w_n(K) t^n} + \sum_{odd\ n}{w_n(K) t^n}$.  Since the sum of the coefficients is 0, by Corollary \ref{C:factor}, $\sum_{even\ n}{w_n(K)} = -\sum_{odd\ n}{w_n(K)} = -ow(K)$.  This means that $W_K(-1) = -2ow(K)$.  So, by Theorem \ref{T:alexwrithe}:
$$2ow(K) = -W_K(-1) = \Delta'_0(K)(-1,-1) = (1+1)(1+1)\overline{\Delta}_0(K)(-1,-1) = 4\overline{\Delta}_0(K)(-1,-1)$$
Hence $ow(K) = 2\overline{\Delta}_0(K)(-1,-1)$.
\end{proof}

\section{Second order writhe polynomial} \label{S:2writhe}

Given a virtual knot diagram $D$ (equivalently, a Gauss diagram), let $S_1$ and $S_2$ be the sets of pairs of chords in configurations $A$ and $B$ (as in Figure \ref{F:configuration}).  Then we define $V_D(t)$ by 
\begin{align*}
V_D(t) = &\frac{Wr(D)(Wr(D)+1)}{2} + \sum_c{\e(c)t^{Ind(c)}\left[{LO(c) - \left(\frac{1+\e(c)}{2}\right)}\right]} \\
&+ \sum_{\{c_i, c_j\} \in S_1}{\e(c_i)\e(c_j)t^{Ind(c_i)+Ind(c_j)}} - \sum_{\{c_i, c_j\} \in S_2}{\e(c_i)\e(c_j)t^{Ind(c_i)+Ind(c_j)}}
\end{align*}

In our next proposition, we analyze how $V_D(t)$ behaves under Reidemeister moves.

\begin{prop} \label{P:writhereid}
Let $K$ be a virtual knot with diagram $D$, and $D'$ the result of applying one of the moves $I_a, I_b, II_a, III_a$ from Figure \ref{F:genreid}.  Then \begin{itemize}
	\item[$I_a$:] $V_{D'}(t) =V_D(t)$
	\item[$I_b$:] $V_{D'}(t) = V_D(t) - W_K(t)$
	\item[$II_a$:] $V_{D'}(t) = V_D(t)$
	\item[$III_a$:] $V_{D'}(t) = V_D(t)$
\end{itemize}
\end{prop}
\begin{proof}
Moves $I_a$ and $I_b$ each add a chord $c$ to the Gauss diagram for $D$, as shown in Figure \ref{F:gaussreid1}.  In both cases, the writhe of the diagram increases by 1, so the first term of $V_D$ increases by $Wr(D)+1$, as shown:
$$\frac{(Wr(D)+1)(Wr(D)+2)}{2} = \frac{Wr(D)(Wr(D)+1) + 2(Wr(D)+1)}{2} = \frac{Wr(D)(Wr(D)+1)}{2} + (Wr(D) + 1).$$

\begin{figure}[htbp]
\begin{center}
\scalebox{1}{\includegraphics{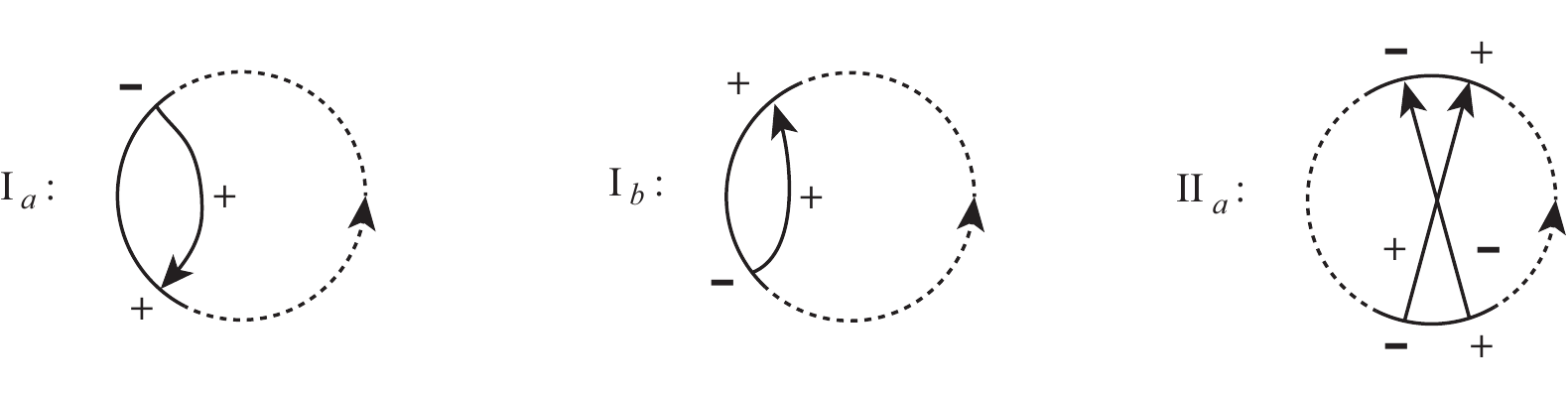}}
\end{center}
\caption{Gauss diagrams for moves $I_a, I_b, II_a$.  No other endpoints are on the solid arcs.}
\label{F:gaussreid1}
\end{figure}

For move $I_a$, $RO(c) = RU(c) = 0$, $LU(c) = Wr(D)$ and $LO(c) = -Wr(D)$, so $Ind(c) = 0$.  So in the first sum in $V_D$ there is a new term $t^0(-Wr(D) - 1) = -(Wr(D)+1)$, which exactly cancels the increase in the first term.  For each other chord $c_i$, both endpoints of chord $c$ are on the same side of $c_i$, so $Ind(c_i)$ doesn't change.  If the pair $\{c, c_i\}$ is in $S_1$, then $c$ is to the left of $c_i$, so $LO(c_i)$ decreases by 1.  So for each such $c_i$, we subtract $\e(c_i)t^{Ind(c_i)}$.  On the other hand, in the second sum, for each $c_i$ with $\{c, c_i\}$  in $S_1$, we add $\e(c_i)t^{Ind(c)+Ind(c_i)} = \e(c_i)t^{Ind(c_i)}$, so these changes also cancel.  Since there are no chords in configuration $B$ with chord $c$, there is no change in the last sum.  So $V_{D'} = V_D$ for move $I_a$.

For move $I_b$, $LO(c) = LU(c) = 0$, so the new term in the first sum is $t^0(0-1) = -1$.  If a chord $c_i$ is parallel to $c$ (i.e. not an alternating configuration), then $c$ is to the left of $c_i$, and $LO(c_i)$ decreases by one, decreasing $V_D$ by $\e(c_i)t^{Ind(c_i)}$.  On the other hand, if the chord $c_i$ {\it is} in an alternating configuration, then $\{c, c_i\}$ is in $S_2$, and we again subtract $\e(c_i)t^{Ind(c_i)}$.  Hence,
$$V_{D'}(t) = V_D(t) + (Wr(D) + 1) - 1 - \sum_{c_i}{\e(c_i)t^{Ind(c_i)}} = V_D(t) - W_K(t)$$

Now we consider move $II_a$.  This move adds two chords $c_1$ and $c_2$ of opposite sign, as shown in Figure \ref{F:gaussreid1} ($c_1$ is the positive chord).  So there is no change to the writhe.  Observe that the two chords have the same index, but $LO(c_2) = LO(c_1)-1$.  So the total contribution to the first sum is
$$t^{Ind(c_1)}(LO(c_1) - 1) - t^{Ind(c_2)}(LO(c_2)) = t^{Ind(c_1)}(LO(c_1) - 1 - (LO(c_1) - 1)) = 0.$$
Finally, since any other chord $c_i$ which alternates with $c_1$ also alternates with $c_2$ in the same configuration, and vice versa, the contributions to the final two sums also cancel.  So $V_{D'} = V_D$ for move $II_a$.

Finally, we look at move $III_a$.  In this case, we are not adding or removing crossings, but rearranging their order along the knot.  Depending on how the remaining arcs of the knot are drawn, there are two possibilities for the Gauss diagram before and after the move, shown in Figure \ref{F:gaussreid2}.  Since no crossings are added or removed, the writhe of the diagram is unchanged.

\begin{figure}[htbp]
\begin{center}
\scalebox{.8}{\includegraphics{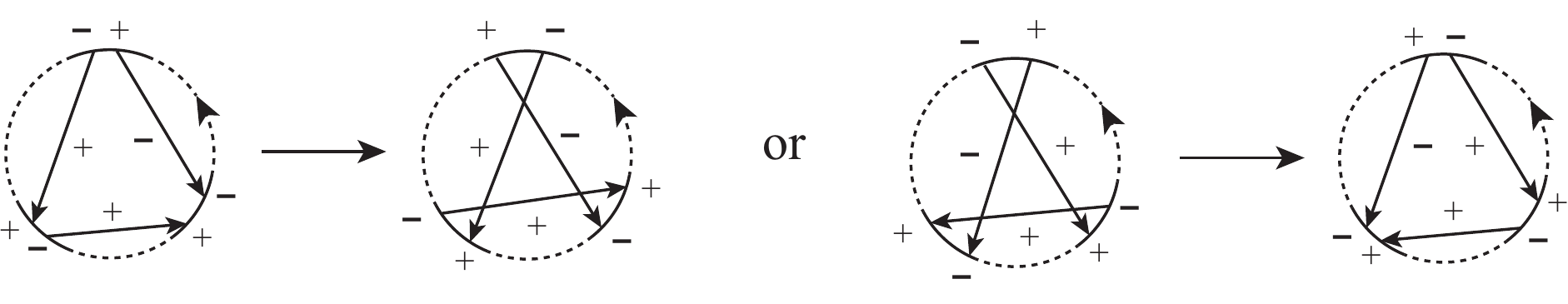}}
\end{center}
\caption{Gauss diagrams for move $III_a$.  No other endpoints are on the solid arcs.}
\label{F:gaussreid2}
\end{figure}

Suppose that $c_1$ and $c_2$ are the positive crossings involved in the $III_a$ move, and $c_3$ is the negative crossing.  The indices of all three crossings are preserved by the move.  $LO(c_1)$ and $LO(c_2)$ are also preserved by the move, but $LO(c_3)$ is reduced by 1.  This adds $(-1)t^{Ind(c_3)}(-1) = t^{Ind(c_3)}$ to the first sum in $V_D$.  The terms for other crossings are unchanged.

Depending on which Gauss diagram in Figure \ref{F:gaussreid2} is considered, a move $III_a$ either removes a configuration of type $A$ or adds a configuration of type $B$ (in both cases involving chords $c_1$ and $c_2$).  In either case, we subtract $t^{Ind(c_1)+Ind(c_2)}$ from $V_D$.  So $V_{D'}(t) = V_D(t) + t^{Ind(c_3)} - t^{Ind(c_1)+Ind(c_2)}$.  But notice that $Ind(c_1) + Ind(c_2) = Ind(c_3)$, so $V_{D'}(t) = V_D(t)$.
\end{proof}

Hence $V_D$ yields a well-defined invariant of the virtual knot $K$, modulo $W_K(t)$; we denote the residue class by $V_K(t)$.  We call $V_K(t)$ the {\em second-order writhe polynomial} for the knot.  Our next theorem shows that this invariant can, like the writhe polynomial, be obtained from the Alexander polynomial.  Consider a virtual knot $K$ with diagram $D$.  The Alexander polynomial for $D$ is $\Delta_0(D)(u,v) = (1-uv)\Delta'_0(D)(u,v)$.  From Theorem \ref{T:alexwrithe}, $W_K(u) = -\Delta'_0(D)(u,u^{-1})$.  Since $\Delta'_0(D)(u,v) + W_K(u) = \Delta'_0(D)(u,v) - \Delta'_0(D)(u,u^{-1})$ is trivial when $v = u^{-1}$ (i.e. when $uv = 1$), it has a factor of $1-uv$.  Define $\Phi(D)(u,v)$ by $\Delta'_0(D)(u,v) + W_K(u) = (1-uv)\Phi(D)(u,v)$.

\begin{thm} \label{T:alexwrithe2}
Let $K$ be a virtual knot with diagram $D$.  Then $V_D(t) = \Phi(D)(t, t^{-1})$.
\end{thm}
\begin{proof}
Recall from the proof of Theorem \ref{T:alexwrithe} that
$$\Delta'_0(D)(u,v) = (uv)^{-Wr(D)}\left(\sum_{i=0}^{Wr(D)-1}{(uv)^i}\right) + f_1(u,v) + (1-uv)f_2(u,v) + \cdots + (1-uv)^{n-1}f_n(u,v)$$
When we add $W_K(u)$, we get
\begin{align*}
\Delta'_0(D)(u,v) + W_K(u) = &\left((uv)^{-Wr(D)}\left(\sum_{i=0}^{Wr(D)-1}{(uv)^i}\right) - Wr(D)\right) + \left(f_1(u,v) + \sum_c{\e(c)u^{Ind(c)}}\right) \\
&+ (1-uv)f_2(u,v) + \cdots + (1-uv)^{n-1}f_n(u,v)
\end{align*}
Looking at the first term, we find:
\begin{align*}
(uv)^{-Wr(D)}\left(\sum_{i=0}^{Wr(D)-1}{(uv)^i}\right) &- Wr(D) = (uv)^{-Wr(D)}\sum_{i=0}^{Wr(D)-1}{\left((uv)^{i} - (uv)^{Wr(D)}\right)} \\
&= (uv)^{-Wr(D)}(1-uv)\sum_{i=0}^{Wr(D)-1}{(uv)^i\left(1 + uv + \cdots + (uv)^{Wr(D)-i-1}\right)} \\
&= (uv)^{-Wr(D)}(1-uv)\sum_{i=0}^{Wr(D)-1}\sum_{j=0}^{Wr(D)-i-1}{(uv)^{i+j}}
\end{align*}
In the second term, we use Proposition \ref{P:alexdet1} to find:
\begin{align*}
f_1(u,v) + \sum_c{\e(c)u^{Ind(c)}} &= \sum_c{-\e(c)(uv)^{-(1+\e(c))/2}u^{-LU(c)}v^{LO(c)}} + \sum_c{\e(c)u^{Ind(c)}} \\
&= \sum_c{\e(c)u^{Ind(c)}\left(1-(uv)^{-(1+\e(c))/2}u^{LO(c)}v^{LO(c)}\right)} \\
&= \sum_c{\e(c)u^{Ind(c)}\left(1-(uv)^{LO(c)-(1+\e(c))/2}\right)} \\
&= (1-uv)\sum_c{\e(c)u^{Ind(c)}\sum_{i=0}^{LO(c)-(1+\e(c))/2 - 1}{(uv)^i}} 
\end{align*}
So then $\Phi(D)(u,v)$ is
\begin{align*}
\Phi(D)(u,v) = &(uv)^{-Wr(D)}\sum_{i=0}^{Wr(D)-1}\sum_{j=0}^{Wr(D)-i-1}{(uv)^{i+j}} + \sum_c{\e(c)u^{Ind(c)}\sum_{i=0}^{LO(c)-(1+\e(c))/2 - 1}{(uv)^i}} \\
&+ f_2(u,v) + (1-uv)f_3(u,v) + \cdots + (1-uv)^{n-2}f_n(u,v)
\end{align*}
Now we set $t = u = v^{-1}$ and use Proposition \ref{P:alexdet1} to obtain
\begin{align*}
\Phi(D)(t, t^{-1}) &= \sum_{i=0}^{Wr(D)-1}\sum_{j=0}^{Wr(D)-i-1}{1} + \left(\sum_c{\e(c)t^{Ind(c)}\sum_{i=0}^{LO(c)-(1+\e(c))/2 - 1}{1}}\right) + f_2(t, t^{-1}) \\
&= \sum_{i=0}^{Wr(D)-1}{(Wr(D)-i)} + \sum_c{\e(c)t^{Ind(c)}\left[LO(c)-\left(\frac{1+\e(c)}{2}\right)\right]} \\
&\qquad + \sum_{\{c_1, c_2\}\in S_1}{\e(c_i)\e(c_2)t^{-LU(c_1) + RU(c_2) + \e_2}t^{-LO(c_1) + RO(c_2) - \e_2}} \\
&\qquad - \sum_{\{c_1, c_2\}\in S_2}{\e(c_i)\e(c_2)t^{-LU(c_1) - LU(c_2)}t^{-LO(c_1) - LO(c_2)}} \\
&= \frac{Wr(D)(Wr(D)+1)}{2} + \sum_c{\e(c)t^{Ind(c)}\left[LO(c)-\left(\frac{1+\e(c)}{2}\right)\right]} \\
&\qquad + \sum_{\{c_1, c_2\}\in S_1}{\e(c_i)\e(c_2)t^{Ind(c_1)+Ind(c_2)}} - \sum_{\{c_1, c_2\}\in S_2}{\e(c_i)\e(c_2)t^{Ind(c_1)+Ind(c_2)}} \\
&= V_D(t)
\end{align*}
\end{proof}

We could use the Alexander polynomial to derive higher-order writhe invariants (each defined modulo the greatest common divisor of the lower-order invariants), but the formulas will quickly become unwieldy.  In the Appendix, we have listed $W_K(t)$ and $V_K(t)$ for all virtual knots with at most 4 real crossings.

\section{Applications} \label{S:applications}

\subsection{Virtual crossing number} \label{SS:vc}

The {\em virtual crossing number} of a virtual knot $K$, denoted $vc(K)$, is the minimum, over all diagrams of $K$, of the number of virtual crossings in the diagram.  If $K$ is classical, then $vc(K) = 0$.  In this section we will prove that the breadth of the writhe polynomial is a lower bound for the virtual crossing number.

Boden et. al. \cite{bo} defined a {\em virtual Alexander polynomial} $H_K(s, t, q)$ for a virtual knot $K$, derived from a {\em virtual knot group}.  Their version of the virtual knot group included relations at virtual crossings (involving the variable $q$), and they showed:

\begin{lem} \label{L:qwidth} \cite[Theorem 3.4]{bo}
Let $K$ be a virtual knot with virtual Alexander polynomial $H_K(s,t,q)$.  Then
$$q{\rm -width\ of\ }H_K(s,t,q) \leq 2vc(K)$$
where the $q$-width is the difference between the largest and smallest powers of $q$ in the polynomial.
\end{lem}

They also determined the relationship between the virtual Alexander polynomial and the polynomial $\Delta_0(K)$ (which they called the {\em generalized Alexander polynomial}):

\begin{lem} \label{L:virtgen} \cite{bo}
For any virtual knot $K$, up to normalization by multiplication by $\pm s^a t^b q^c$,
$$H_K(s,t,q) = H_K(sq^{-1}, tq, 1) = \Delta_0(K)(sq^{-1}, tq)$$
\end{lem}

Both $H_K(s,t,q)$ and $\Delta_0(K)(sq^{-1}, tq)$ are divisible by $1-st$.  Let $H_K'(s,t,q)$ and $\Delta_0'(K)(sq^{-1}, tq)$ denote the quotients.  Note that the $q$-width of $H_K'$ is the same as the $q$-width of $H_K$.  So then
\begin{align*}
q{\rm -width}(H_K(s,t,q)) &= q{\rm -width}(H_K'(s,t,q)) \\
&\geq q{\rm -width}(H_K'(1,1,q)) \\
&= q{\rm -width}(\Delta_0'(K)(q^{-1}, q)) {\rm \ by\ Lemma\ \ref{L:virtgen}}\\
&= q{\rm -width}(W_K(q)) {\rm \ by\ Theorem\ \ref{T:alexwrithe}.}
\end{align*}

Combining this with Lemma \ref{L:qwidth}, we have shown

\begin{thm} \label{T:vcwrithe}
If $K$ is a virtual knot, then
$$width(W_K(t)) \leq 2vc(K)$$
where $width(W_K(t))$ is the difference between the largest and smallest powers of $t$ in the polynomial.
\end{thm}

\subsection{Forbidden number} \label{SS:forbidden}

There are two additional Reidemeister-like moves for virtual knots, known as the \emph{forbidden moves}, illustrated in Figure ~\ref{F:Fmoves}. Move $FO$ moves a strand of the diagram ``over'' a virtual crossing, while move $FU$ moves a strand ``under'' a virtual crossing. 

\begin{figure}[htbp]
\begin{center}
\includegraphics{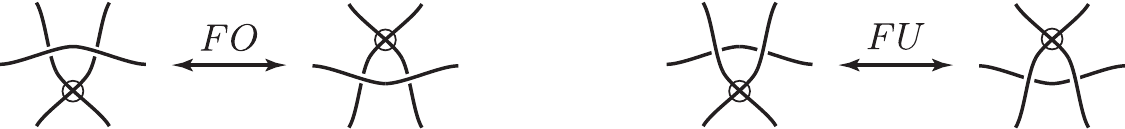}
\end{center}
\caption{Forbidden moves.}
\label{F:Fmoves}
\end{figure}

Unlike the other virtual Reidemeister moves, the forbidden moves {\em do} change the Gauss diagram of a virtual knot.  The move $FO$ has the effect of switching the tails of two chords $c_1$ and $c_2$ in a Gauss diagram, while the move $FU$ switches the heads, as shown in Figure ~\ref{F:FGauss}.  Nelson \cite{ne} notes that we can move a head past a tail by using one forbidden move of each type.

\begin{figure}[htbp]
\begin{center}
{\includegraphics{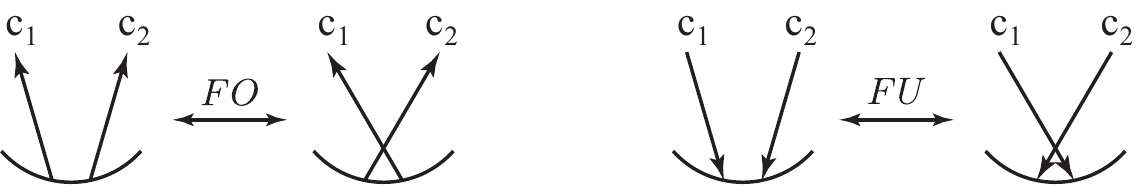}}
\end{center}
\caption{Forbidden moves on Gauss diagrams.}
\label{F:FGauss}
\end{figure}

Taken together, the forbidden moves are an unknotting operation \cite{gpv, kan, ne}.  The {\em forbidden number} of a virtual knot is the minimum number of forbidden moves required to transform the knot into the unknot (with the minimum taken over all diagrams of the knot).  Crans, Ganzell and the author \cite{cgm} found an upper bound for the forbidden number based on the crossing number, and a lower bound based on Cheng's odd writhe polynomial \cite{ch}.  Sakurai \cite{sak} found another (usually stronger) lower bound using Henrich's polynomial \cite{he}.  Since both the odd writhe polynomial and Henrich's polynomial are derived from the writhe polynomial $W_K(t)$, we can find a lower bound as strong (or stronger) as either of these using the writhe polynomial.

From Figure \ref{F:FGauss}, it is clear that the $FO$ and $FU$ moves subtract $\e(c_2)$ from $Ind(c_1)$ and add $\e(c_1)$ to $Ind(c_2)$ (going from the diagram on the left to the one on the right).  The indices of all other chords are unchanged.  So if $K$ is a virtual knot, and $K'$ is the result of applying a forbidden move to crossings $c_1$ and $c_2$, then
\begin{align*}
W_K(t) - W_{K'}(t) &= \pm[\e(c_1)(t^{Ind(c_1)-\e(c_2)} - t^{Ind(c_1)}) + \e(c_2)(t^{Ind(c_2)+\e(c_1)} - t^{Ind(c_2)})] \\
&= \pm[\e(c_1)t^{Ind(c_1)}(t^{-\e(c_2)} - 1) + \e(c_2)t^{Ind(c_2)}(t^{\e(c_1)} - 1)]\\
&= \pm t^{Ind(c_1)}(t^{\pm 1} - 1) \pm t^{Ind(c_2)}(t^{\pm 1} - 1) \\
&= \pm (t-1) \cdot \left\{ \begin{matrix} t^{Ind(c_1)} \pm t^{Ind(c_2)} \\ t^{Ind(c_1)-1} \pm t^{Ind(c_2)} \\ t^{Ind(c_1)} \pm t^{Ind(c_2)-1} \\ t^{Ind(c_1)-1} \pm t^{Ind(c_2)-1} \end{matrix} \right.
\end{align*}
If we define $W'_K(t)$ by $W_K(t) = (t-1)W'_K(t)$, then we have shown:

\begin{thm} \label{T:Wforbidden}
Suppose $K$ is a virtual knot, and $W'_K(t) = \sum{b_i t^i}$.  If $K'$ is the result of performing a forbidden move, and $W'_{K'}(t) = \sum{c_i t^i}$, then $\left\vert \left(\sum{\vert b_i \vert} - \sum{\vert c_i \vert} \right)\right\vert \leq 2$.  In particular, the forbidden number of $K$ is bounded below by $\frac{1}{2}\sum{\vert b_i \vert}$.
\end{thm}

The lower bound on the forbidden number given by $W_K$ is at least as good as those derived from Henrich's polynomial or the odd writhe polynomial, and is sometimes significantly better.  For example, consider the virtual knot $K$ shown (with its Gauss diagram) in Figure \ref{F:Fexample}.  $W_K(t) = t^{-4} - 2t^{-2} + 2t^2 - t^4 = (t-1)(1 - t^{-4} - t^{-3} + t^{-2} + t^{-1} + t - t^2 - t^3)$.  Both the Henrich and odd writhe polynomials are 0, so give us no information about the forbidden number.  But from Theorem \ref{T:Wforbidden} we get that the forbidden number of $K$ is at least 4.  (It is not hard to see that $K$ can be unknotted using 9 forbidden moves; we do not know whether it can be done with fewer.)

\begin{figure}[htbp]
\begin{center}
\scalebox{.8}{\includegraphics{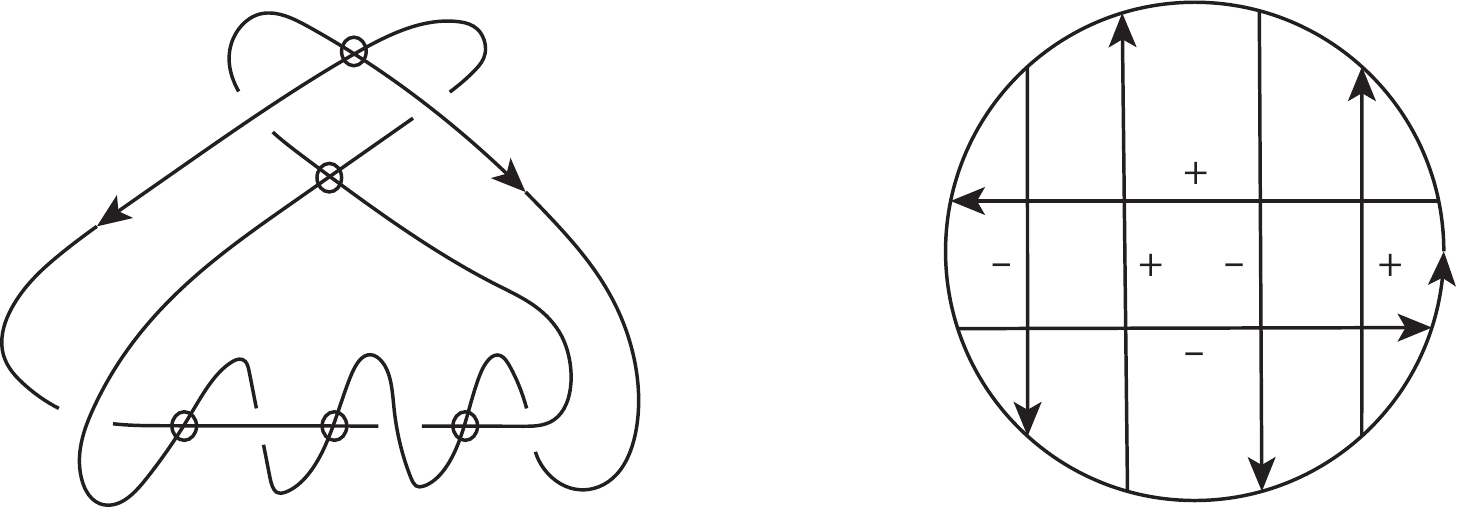}}
\end{center}
\caption{$K =$ O1-U2+O3-U1-O4+U5-O6+U3-O2+U6+O5-U4+.}
\label{F:Fexample}
\end{figure}

The second order writhe polynomial $V_K(t)$ can also give us information about the forbidden number.  Since the pair of chords altered by a forbidden move aren't in an alternating configuration either before or after the move, the only part of $V_K(t)$ changed is the first sum, and the only parts of the sum which change are the terms corresponding to $c_1$ and $c_2$.  So at most four coefficients of the polynomial change, those corresponding to $t^{Ind(c_1)}$, $t^{Ind(c_1)-\e(c_2)}$, $t^{Ind(c_2)}$, and $t^{Ind(c_2)+\e(c_1)}$.  The amount by which the coefficients change is determined by the indices of the crossings, and can be relatively large.  However, we have shown

\begin{prop} \label{P:Vforbidden}
Suppose $K$ is a virtual knot with forbidden number 1.  Then $V_K(t)$ can be written with at most four terms, of which at most two involve even powers of $t$ and at most two involve odd powers of $t$.
\end{prop}

For example, consider knot 4.2 from Jeremy Green's table of virtual knots \cite{gr}, shown in Figure \ref{F:4.2}.  The writhe polynomial for this knot is trivial, so gives no lower bound for the forbidden number.  The second order writhe polynomial is $V_K(t) = 2 + u^{-2} - 2u^{-1} - 2u + u^2$; since the writhe polynomial is trivial, the second order polynomial is well-defined.  $V_K(t)$ has five terms, so the forbidden number is at least 2.  However, the Gauss code for this knot is O1-O2-U1-U2-O3+O4+U3+U4+, and it is easy to see that it can be unknotted with two forbidden moves (one move to uncross chords 1 and 2, and a second to uncross chords 3 and 4).  So the forbidden number for knot 4.2 is exactly 2.

\begin{figure}[htbp]
\begin{center}
\scalebox{2}{\includegraphics{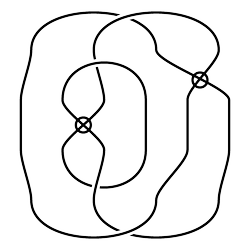}}
\end{center}
\caption{Virtual knot 4.2 (from Green's table of virtual knots \cite{gr}).}
\label{F:4.2}
\end{figure}

\subsection{Positive reflection mutations} \label{SS:mutation}

Given a virtual knot $K$, a {\em Conway mutation} of $K$ is the process of cutting out a tangle (cutting the knot $K$ in four places), transforming the tangle by a horizontal flip, vertical flip, or $180^\circ$ rotation, and gluing it back in.  The rotation is {\em positive} if the orientation of the strands of the tangle are the same before and after the mutation.  A {\em positive rotation} is a positive mutation which rotates the tangle, while a {\em positive reflection} flips the tangle.  Folwaczny and Kauffman \cite{fk} showed that the writhe polynomial could distinguish some pairs of {\em positive rotation} mutants, but could {\em not} distinguish {\em positive reflection} mutants.  We will show that the second order writhe polynomial {\em can} sometimes distinguish positive reflection mutants.  In fact, we will give an infinite family of pairs $(K, MK)$ of a virtual knot and its positive reflection mutant which are distinguished by the second order writhe polynomial.

\begin{figure}[htbp]
\begin{center}
\scalebox{.8}{\includegraphics{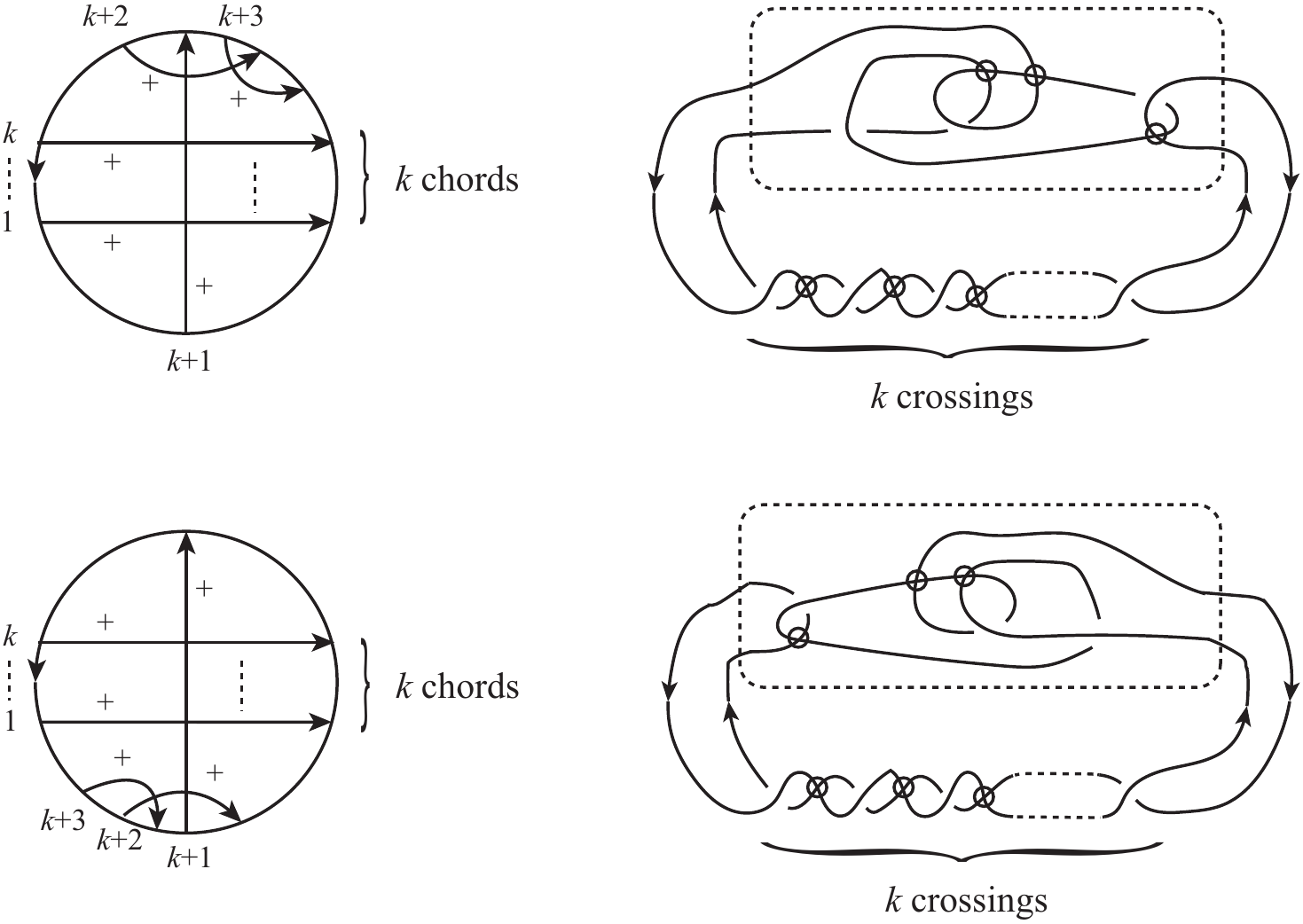}}
\end{center}
\caption{{\bf Top:} Gauss diagram and knot diagram for virtual knot $K$. {\bf Bottom:} Gauss diagram and knot diagram for positive reflection mutant $MK$.  All crossings are positive in both knots.}
\label{F:mutant}
\end{figure}

Consider the virtual knot $K$ shown on the top in Figure \ref{F:mutant} and its positive reflection mutant $MK$ shown on the bottom of Figure \ref{F:mutant}.  The table below shows the index and left over-index for each crossing.

\medskip
\begin{center}
\begin{tabular}{|c|c|c|c|c|}
\hline
& \multicolumn{2}{c |}{$K$} & \multicolumn{2}{c |}{$MK$} \\ \hline
crossing & $LO$ & $Index$ & $LO$ & $Index$ \\ \hline
$1 \leq i \leq k$ & $i - k - 2$ & $-1$ & $i - k$ & $-1$ \\ \hline
$k+1$ & $-k-1$ & $k+1$ & $-k-2$ & $k+1$ \\ \hline
$k+2$ & $-1$ & $0$ & $-k-1$ & $0$ \\ \hline
$k+3$ & $0$ & $-1$ & $-k-1$ & $-1$ \\ \hline
\end{tabular}
\end{center}
\medskip

\noindent Since the indices of the crossings are all the same, the knots have the same writhe polynomial (as expected from \cite{fk}).  The writhe polynomial is $W_K(t) = W_{MK}(t) = -k-2 + kt^{-1} + t^{k+1}$.  However, the second order writhe polynomials are not the same.  Note that, in both $K$ and $MK$, the only pair of alternating chords is $k+1$ and $k+3$.  In $K$ these chords are in configuration $B$, while in $MK$ they are in configuration $A$.
\begin{align*}
V_K(t) &= \dfrac{(k+3)(k+4)}{2} + \left(\sum_c{t^{Ind(c)}(LO(c) - 1)}\right) - t^k \\
& = \dfrac{(k+3)(k+4)}{2} + \left(\sum_{i=1}^k{i-k-3}\right) t^{-1} + (-k-2)t^{k+1} - 2 - t^{-1} - t^k \\
& = \dfrac{(k+3)(k+4)}{2}  - \left(\dfrac{(k+2)(k+3)}{2} - 3\right) t^{-1} - (k+2)t^{k+1} - 2 - t^{-1} - t^k \\
& = \dfrac{k^2+7k+8}{2} - \left(\dfrac{k^2+5k+2}{2}\right) t^{-1} - (k+2)t^{k+1} - t^k
\end{align*}
By a similar computation, we find 
$$V_{MK}(t) = \dfrac{k^2+5k+8}{2} - \left(\dfrac{k^2-k-4}{2}\right) t^{-1} - (k+3)t^{k+1} + t^k$$
The difference is
$$V_K(t) - V_{MK}(t) = k - 3(k+1)t^{-1} + t^{k+1} - 2t^k$$

This is not a multiple of $W_K(t)$, since the writhe polynomial for the two knots does not have a $t^k$ term.  Hence the second order writhe polynomials are not congruent modulo $W_K(t)$.  So the second order writhe polynomial distinguishes $K$ from $MK$.

\newpage
\section{Appendix: Table of writhe polynomials}

These tables list the writhe polynomials $W_K(t)$ and second order writhe polynomials $V_K(t)$ (modulo $W_K(t)$) for all virtual knots with at most four crossings.  The knots were taken from Green's table of virtual knots \cite{gr}.

\begin{center}
\begin{tabular}{| r | r | r |}
\hline
{\bf Knot $K$} & {\bf $W_K(t)$} & {\bf $V_K(t)$} \\ \hline
0.1 & $0$ & $0$ \\ \hline
2.1 & $-2 + t^{-1} + t$ & $t-1$ \\ \hline
3.1 & $1 - t^{-2} + t^{-1} - t$ & $2-2t$\\ \hline
3.2 & $2 - t^{-1} - t$ & $1-t$ \\ \hline
3.3 & $3 - t^{-2} - 2t$ & $3-3t$ \\ \hline
3.4 & $-1 - t^{-2} + 2t^{-1}$& $t^{-1}-1$\\ \hline
3.5 & $2 - t^{-2} - t^2$ & $2-2t^2$\\ \hline
3.6 & $0$ & $0$ \\ \hline
3.7 & $2 - t^{-2} - t^2$ & $2-2t^2$\\ \hline
4.1 & $4 - 2t^{-1} - 2t$ & $- t^{-2} + 2t^{-1} - t^2$ \\ \hline
4.2 & $0$ & $2 + t^{-2} - 2t^{-1} - 2t + t^2$ \\ \hline
4.3 & $-4 + 2t^{-1} + 2t$ & $3 + t^{-2} - 4t^{-1}$ \\ \hline
4.4 & $2 - t^{-1} - t$ & $t - t^2$ \\ \hline
4.5 & $2 - t^{-1} - t$ & $t^{-2} - t^{-1}$ \\ \hline
4.6 & $0$ & $-1 + 2t - t^2$ \\ \hline
4.7 & $4 - 2t^{-1} - 2t$ & $2 - 2t$ \\ \hline
4.8 & $0$ & $0$ \\ \hline
4.9 & $-2 + t^{-1} + t$ & $1 - t^{-1}$ \\ \hline
4.10 & $1 - t^{-2} + t^{-1} - t$ & $2-2t$ \\ \hline
4.11 & $3 - t^{-2} - 2t$ & $4 - t^{-1} - 2t - t^2$ \\ \hline
4.12 & $0$ & $t^{-2} - 2t^{-1} + 2t - t^2$ \\ \hline
4.13 & $0$ & $-1 - t^{-2} + 2t^{-1}$\\ \hline
4.14 & $-t^{-2} + t^{-1} + t - t^2$ & $1 + t - 2t^2$ \\ \hline
4.15 & $3 - t^{-2} - 2t$ & $5 - t^{-1} - 4t$ \\ \hline
4.16 & $0$ & $0$ \\ \hline
4.17 & $1 - t^{-2} + t^{-1} - t$ & $t^{-1} - t$ \\ \hline
4.18 & $-2 + t^{-1} + t$ & $t-1$ \\ \hline
4.19 & $1 - t^{-1} + t - t^2$ & $1 - t^2$ \\ \hline
4.20 & $1 + t^{-2} - 2t^{-1}$& $1 - t^{-1}$\\ \hline
4.21 & $2 - t^{-2} - t^2$ & $1 + t^{-1} - t - t^2$ \\ \hline
4.22 & $1 - t^{-3} + t^{-2} - t$ & $1 + t^{-1} - t$ \\ \hline
4.23 & $-1 + t^{-1} - t + t^2$ & $t^2 - 1$ \\ \hline
4.24 & $1 - t^{-3} + t^{-1} - t^2$ & $2 + t^{-1} - t - 2t^2$ \\ \hline
4.25 & $-4 + 2t^{-1} + 2t$ & $t-t^{-1}$ \\ \hline
4.26 & $t^{-3} - 2t^{-1} + t$ & $-2  - t^{-1} + 3t$\\ \hline
4.27 & $2 - t^{-1} - t$ & $1 - t$ \\ \hline
4.28 & $-2 + t^{-3} - t^{-1} + 2t$ & $-1 - 2t^{-1} + 3t$ \\ \hline
4.29 & $-3 + t^{-2} + 2t$ & $-4 + t^{-1} + 2t + t^2$ \\ \hline
4.30 & $2 - t^{-1} - t$ & $-2 - t^{-2} + 3t^{-1}$ \\ \hline
\end{tabular}
\end{center}

\begin{center}
\begin{tabular}{| r | r | r |}
\hline
{\bf Knot $K$} & {\bf $W_K(t)$} & {\bf $V_K(t)$} \\ \hline
4.31 & $0$ & $1 - 2t + t^2$ \\ \hline
4.32 & $-1 + t^{-2} - t^{-1} + t$ & $-2 + 2t$ \\ \hline
4.33 & $2 - t^{-1} - t$ & $-1 + t^{-1}$ \\ \hline
4.34 & $-1 - t^{-2} + 2t^{-1}$& $-1 + t^{-1}$\\ \hline
4.35 & $-1 + t^{-1} - t + t^2$ & $-1 + t^2$ \\ \hline
4.36 & $-2 + t^{-2} + t^2$ & $1 - t^{-2} - t^{-1} + t$ \\ \hline
4.37 & $-4 + t^{-2} + t^{-1} + t + t^2$ & $-6 + 2t^{-1} + t + 3t^2$ \\ \hline
4.38 & $1 - 2t + t^2$ & $-1 + t$ \\ \hline
4.39 & $1 - t^{-3} + t^{-2} - t$ & $1 + t^{-1} - 2t$ \\ \hline
4.40 & $-2 + t^{-1} + t$ & $-1 + t$ \\ \hline
4.41 & $0$ & $0$ \\ \hline
4.42 & $-1 - t^{-3} + t^{-2} + t^{-1}$ & $-1 + t^{-2}$\\ \hline
4.43 & $4 - 2t^{-1} - 2t$ & $2 - 2t$ \\ \hline
4.44 & $-2 + t^{-1} + t$ & $-1 + t$ \\ \hline
4.45 & $-2 + t^{-3} - t^{-1} + 2t$ & $-3 - t^{-1} + 4t$ \\ \hline
4.46 & $0$ & $2 - t^{-1} - t$ \\ \hline
4.47 & $t^{-3} - 2t^{-1} + t$ & $-2t^{-1}+2t$ \\ \hline
4.48 & $4 - t^{-2} - t^{-1} - t - t^2$ & $1+t^{-1}-2t^2$ \\ \hline
4.49 & $-1 + 2t - t^2$ & $1-t$ \\ \hline
4.50 & $-1 + t^{-2} - t^{-1} + t$ & $-2+2t$ \\ \hline
4.51 & $0$ & $1 - 2t + t^2$ \\ \hline
4.52 & $2 - t^{-1} - t$ & $-1+t^{-1}$ \\ \hline
4.53 & $-4 + 2t^{-1} + 2t$ & $2-2t^{-1}$ \\ \hline
4.54 & $2 - t^{-1} - t$ & $t-t^2$ \\ \hline
4.55 & $0$ & $0$ \\ \hline
4.56 & $0$ & $0$ \\ \hline
4.57 & $1 - t^{-2} + t^{-1} - t$ & $t^{-1}-t$ \\ \hline
4.58 & $0$ & $0$ \\ \hline
4.59 & $0$ & $0$ \\ \hline
4.60 & $-2 + t^{-1} + t$ & $-1+t$ \\ \hline
4.61 & $-2 + t^{-1} + t$ & $-1+t$ \\ \hline
4.62 & $3 - t^{-3} - t - t^2$ & $-2+t-t^2$ \\ \hline
4.63 & $3 - t^{-2} - 2t$ & $-1+2t^{-2}-t^{-1}$ \\ \hline
4.64 & $-t^{-2} + t^{-1} + t - t^2$ & $3t-3t^2$\\ \hline
4.65 & $2 - t^{-2} - t^2$ & $1+t^{-1}-t-t^2$ \\ \hline
4.66 & $1 - t^{-3} + t^{-1} - t^2$ & $2+t^{-1}-t-2t^2$ \\ \hline
4.67 & $1 - t^{-2} + t^{-1} - t$ & $t^{-1}-t$ \\ \hline
4.68 & $0$ & $0$ \\ \hline
4.69 & $2 - t^{-1} - t$ & $1-t$ \\ \hline
\end{tabular}
\end{center}

\begin{center}
\begin{tabular}{| r | r | r |}
\hline
{\bf Knot $K$} & {\bf $W_K(t)$} & {\bf $V_K(t)$} \\ \hline
4.70 & $-1 + t^{-2} - t^{-1} + t$ & $-2+2t$ \\ \hline
4.71 & $0$ & $0$ \\ \hline
4.72 & $0$ & $0$ \\ \hline
4.73 & $-4 + 2t^{-1} + 2t$ & $2-t^{-2}-2t+t^2$ \\ \hline
4.74 & $2 - t^{-1} - t$ & $2-3t+t^2$ \\ \hline
4.75 & $0$ & $0$ \\ \hline
4.76 & $0$ & $0$ \\ \hline
4.77 & $0$ & $0$ \\ \hline
4.78 & $3 - t^{-3} - t - t^2$ & $4-t-3t^2$ \\ \hline
4.79 & $-1 - t^{-3} + t^{-2} + t^{-1}$ & $-1+t^{-2}$\\ \hline
4.80 & $-4 + t^{-3} + 3t$ & $-6+6t$ \\ \hline
4.81 & $2 + t^{-3} - 3t^{-1}$& $3-3t^{-1}$\\ \hline
4.82 & $-4 + t^{-2} + t^{-1} + t + t^2$ & $3-2t^{-2}-t^{-1}$ \\ \hline
4.83 & $2 - t^{-3} + t^{-2} - t^{-1} - t^2$ & $4-t^{-1}-3t^2$\\ \hline
4.84 & $t^{-2} - t^{-1} - t + t^2$ & $1-2t^{-2}+t^{-1}$ \\ \hline
4.85 & $-2 + t^{-2} + t^2$ & $2t^2-2$\\ \hline
4.86 & $2 - t^{-2} - t^2$ & $2-2t^2$ \\ \hline
4.87 & $4 - t^{-3} - t^{-1} - 2t^2$ & $6-t^{-1}-5t^2$ \\ \hline
4.88 & $t^{-3} - 2t^{-2} + t^{-1}$ & $-t^{-2}+t^{-1}$\\ \hline
4.89 & $-4 + 2t^{-2} + 2t^2$ & $2-3t^{-2}+t^2$ \\ \hline
4.90 & $0$ & $0$ \\ \hline
4.91 & $4 - t^{-3} - t^{-1} - t - t^3$ & $6-t^{-1}-2t-3t^3$ \\ \hline
4.92 & $-2 + t^{-3} + t^3$ & $-3+3t^3$ \\ \hline
4.93 & $2 + t^{-3} - 2t^{-2} - t$ & $3-t^{-2}-2t$ \\ \hline
4.94 & $2 - t^{-1} - t$ & $1-t$ \\ \hline
4.95 & $-2 + t^{-3} + t^3$ & $-3+3t^3$ \\ \hline
4.96 & $-2 + t^{-2} + t^2$ & $-2+2t^2$ \\ \hline
4.97 & $t^{-2} - t^{-1} - t^2 + t^3$ & $1 - 3t^{-2} + 2t^{-1}$\\ \hline
4.98 & $0$ & $0$ \\ \hline
4.99 & $0$ & $0$ \\ \hline
4.100 & $-4 + 2t^{-1} + 2t$ & $2-2t^{-1}$ \\ \hline
4.101 & $2 - t^{-3} - t^3$ & $3-3t^3$ \\ \hline
4.102 & $-t^{-3} + t^{-1} + t - t^3$ & $t^{-1}+2t-3t^2$ \\ \hline
4.103 & $-2 + 2t^{-2} - t^{-1} + t^3$ & $3 - 5t^{-2} + 2t^{-1}$\\ \hline
4.104 & $-2 + t^{-3} + t^3$ & $-3+3t^3$ \\ \hline
4.105 & $0$ & $0$ \\ \hline
4.106 & $2 - t^{-2} - t^2$ & $2-2t^2$ \\ \hline
4.107 & $0$ & $-2 + t^{-2} + t^2$ \\ \hline
4.108 & $0$ & $0$ \\ \hline
\end{tabular}
\end{center}

\newpage



\end{document}